\documentclass[12pt,twoside,reqno]{amsart}

\usepackage{amssymb,amsmath,amstext,amsthm,amsfonts,amscd,xcolor}

\usepackage{dsfont}

\usepackage[ansinew]{inputenc} 
\usepackage{graphicx}
\usepackage[mathscr]{eucal}
\usepackage{hyperref}

\newcommand{\R}{\mathbb{R}}
\newcommand{\C}{\mathbb{C}}
\newcommand{\N}{\mathbb{N}}
\newcommand{\Z}{\mathbb{Z}}

\newcommand{\T}{\mathbb{T}}
\newcommand{\SL}{{\rm SL}}
\newcommand{\GL}{{\rm GL}}
\newcommand{\gl}{{\rm Mat}}
\newcommand{\Mat}{{\rm Mat}}

\newcommand{\strip}{\mathscr{A}}

\newcommand{\sabs}[1]{\left| #1 \right|} 
\newcommand{\avg}[1]{\left< #1 \right>} 
\newcommand{\abs}[1]{\bigl| #1 \bigr|} 
\newcommand{\norm}[1]{\lVert#1\rVert} 
\newcommand{\bnorm}[1]{\Bigl\| #1\Bigr\|} 

\newcommand{\intpart}[1]{\lfloor #1 \rfloor} 
\newcommand{\transl}{{T}} 
\newcommand{\less}{\lesssim}
\newcommand{\more}{\gtrsim}
\newcommand{\ep}{\epsilon} 
\newcommand{\ka}{\kappa} 
\newcommand{\om}{\omega} 
\newcommand{\GLmR}{\GL(m, \R)}

\newcommand{\cocycles}{\mathcal{C}}

\newcommand{\params}{\mathscr{P}}

\newcommand{\nzerobar}{\underline{n_0}}

\newcommand{\dev}{\ep}                
\newcommand{\devf}{\underline{\dev}}  
\newcommand{\mes}{\iota}              
\newcommand{\mesf}{\underline{\mes}}  

\newcommand{\An}[1]{A^{({#1})}}  
\newcommand{\Bn}[1]{B^{({#1})}}  

\newcommand{\comp}{^{\complement}}

\newcommand{\ind}{\mathds{1}}

\newcommand{\B}{\mathscr{B}}                                                                                                   
\newcommand{\Bbar}{\bar{\mathscr{B}}}

\newcommand{\dist}{{\rm dist}}

\newcommand{\filt}{F}

\newcommand{\rgap}{{\rm gr}}

\newcommand{\Filt}{\mathfrak{F}}

\newcommand{\Filatp}[1]{\mathfrak{F}_{\supset{#1}}}

\newcommand{\dec}{E_{\cdot}}
\newcommand{\Dec}{\mathfrak{D}}

\newcommand{\Decatp}[1]{\mathfrak{D}_{\supset{#1}}}

\makeatletter
\newsavebox{\@brx}
\newcommand{\llangle}[1][]{\savebox{\@brx}{\(\m@th{#1\langle}\)}%
  \mathopen{\copy\@brx\mkern2mu\kern-0.9\wd\@brx\usebox{\@brx}}}
\newcommand{\rrangle}[1][]{\savebox{\@brx}{\(\m@th{#1\rangle}\)}%
  \mathclose{\copy\@brx\mkern2mu\kern-0.9\wd\@brx\usebox{\@brx}}}
\makeatother

\newcommand{\qpcmat}[1]{C^{\om}_{r} (\T^{#1}, \Mat (m, \R))}

\newcommand{\analyticf}[1]{C^{\om}_{r} (\T^{#1}, \R)}
\newcommand{\cocyclesone}{C^{\om}_{r} (\T, \Mat_m(\R))}

\newcommand{\normr}[1]{\lVert#1\rVert_r} 
\newcommand{\LE}[1]{L^{({#1})}}  
\newcommand{\normtwo}[1]{
{\left\vert\kern-0.25ex\left\vert\kern-0.25ex\left\vert #1 
    \right\vert\kern-0.25ex\right\vert\kern-0.25ex\right\vert} } 
\newcommand{\normz}[1]{\lVert#1\rVert_0} 

\newcommand{\norml}[2]{\lVert#1\rVert_{#2}} 
\newcommand{\Lp}[1]{L^{#1} (\T)}

\newcommand{\ga}{\gamma}
\newcommand{\li}{(l_1, l_2)}
\newcommand{\x}{(x_1, x_2)}
\newcommand{\lli}{(l_1, l_2)}
\newcommand{\omm}{(\om_1, \om_2)}
\newcommand{\un}[2]{u^{({#1})}_{{#2}}}   

\newcommand{\Zm}{\mathscr{Z}}

\newcommand{\umeas}{\mathscr{S}}
\newcommand{\ldtmeas}{\mathscr{C}}

\theoremstyle{plain}
\newtheorem{theorem}{Theorem}[section]
\newtheorem{proposition}{Proposition}[section]

\newtheorem{lemma}[proposition]{Lemma}
\newtheorem{definition}{Definition}[section]

\numberwithin{equation}{section}

\newtheorem{remark}{Remark}[section]

\title[Quasi-periodic cocycles with singularities]{Large deviations for quasi-periodic cocycles with singularities}

\date{}

\begin{document}

\author[P. Duarte]{Pedro Duarte}
\address{Departamento de Matem\'atica and CMAFIO\\
Faculdade de Ci\^encias\\
Universidade de Lisboa\\
Portugal 
}
\email{pmduarte@fc.ul.pt}

\author[S. Klein]{Silvius Klein}
\address{ Department of Mathematical Sciences\\
Norwegian University of Science and Technology (NTNU)\\
Trondheim, Norway\\ 
and IMAR, Bucharest, Romania }
\email{silvius.klein@math.ntnu.no}

\begin{abstract}
We derive large deviations type (LDT) estimates for linear cocycles over an ergodic multifrequency torus translation. These models are called quasi-periodic cocycles. We make the following assumptions on the model: the translation vector satisfies a generic Diophantine condition, and the fiber action is given by a matrix valued analytic function of several variables which is not identically singular. The LDT estimates obtained here depend on some uniform measurements on the cocycle. 

Our general results derived in~\cite{LEbook} 
regarding the continuity properties of the Lyapunov exponents (LE) and of the Oseledets filtration and decompositions are then applicable, and we obtain local weak-H\"older continuity of these quantities in the presence of gaps in the Lyapunov spectrum. The main new feature of this work is allowing a cocycle depending on several variables to have singularities, i.e. points of non invertibility.  This requires a careful analysis of the set of zeros of certain analytic functions of several variables and of the singularities (i.e. negative infinity values) of pluri-subharmonic functions related to the iterates of the cocycle.

A refinement of this method in the one variable case leads to a stronger LDT estimate and in turn to a stronger, nearly-H\"older modulus of continuity of the LE, Oseledets filtration and Oseledets decomposition. 

This is a draft of a chapter  in our forthcoming research monograph~\cite{LEbook}.  
\end{abstract}

\maketitle

\section{Introduction and statements}
\label{qp_intro}

We define the quasi-periodic cocycles  and describe our assumptions on them. We then formulate the main statements and  relate them to recent results for similar models.

\subsection{Description of the model}
\label{qp_model}
Let $\T = \R/\Z$ be the one variable torus, which we may regard as the unit circle in the complex plane. We use the notation $e (x) = e^{2 \pi i x}$ and in fact we write $f (x)$ instead of $f (e (x))$ whenever  $f$ is a function on $\T$.  

Throughout this paper, $a \less b$ will stand for $a \le C \,b$ for some context-universal constant $C$, which may be discarded from subsequent estimates. 

\smallskip

Let  $\T^d  = (\R/\Z)^d$ be the torus with $d \ge 1$ variables. We denote by $\sabs{ \cdot}$ the Haar measure on $\T^d$.
Let $\transl x = x + \om$ be the translation on the torus $\T^d $ by the vector $\om = (\om_1, \om_2, \ldots, \om_d)$. We assume that $1, \om_1, \om_2, \ldots, \om_d$ are rationally independent, hence $\transl$ is ergodic. The map $T$  defines the base dynamics and it is assumed fixed.

Let  $A \colon \T^d \to \Mat (m, \R)$ be a matrix valued real {\em analytic} function.

The pair $(\transl, A)$, acting on the vector bundle $\T^d \times \R^m$ by $(x, v) \mapsto (\transl x, \, A (x) v)$ is called an analytic, quasi-periodic linear cocycle (quasi-periodic because of the base dynamics, linear due to the linear fiber action, and analytic due to the analytic dependence on the base point). As before, the frequency vector $\om$ will be fixed, hence we identify the cocycle with its fiber action given my the function $A (x)$.

\vspace{4pt}

In order to be able to treat occurrences of small denominators, the translation vector will be assumed to satisfy a generic {\em Diophantine condition}:
\begin{equation}\label{DC}
\norm{k \cdot \om} \ge \frac{t}{\sabs{k}^{d+\delta_0}}
\end{equation}
for some $t > 0$, $\delta_0 > 0$  and for all $k \in \Z^d \setminus \{0\}$, where  for any real number $x$ we denote $\norm{x}:=\min_{k\in\Z} \abs{x-k}$.

Note that if $\delta_0$ is fixed and if for every $t > 0$, $\rm{DC}_t$ denotes the set of frequency vectors satisfying the condition \eqref{DC} above,  then the set $\rm{DC} := \cup_{t > 0} \, \rm{DC}_t$ has full measure.  

\vspace{4pt}

Since $A (x)$ is analytic on $\T^d$, it has an extension $A (z)$ to  $\strip_r^d = \strip_r \times \ldots \times \strip_r$, where $\strip_r = \{ z \in \C \colon 1 - r < \abs{z} < 1 + r \}$ is the annulus of width $2 r$.
Note that the iterates $\An{n} (x) := A (x + (n-1) \, \om)  \, \ldots \,A (x + \om) \,A (x)$ of the cocycle are also analytic functions on $\strip_r^d$.

We denote by $\analyticf{d}$ the Banach space of real valued analytic functions on $\strip_r^d$ with continuous extension up to the boundary and norm $\norm{f}_r := \sup_{z \in \strip_r^d} \, \abs{f (z) }$.

For every integer $m \ge 1$, let $\qpcmat{d}$ be the vector space of matrix valued analytic functions on $\strip_r^d$, with continuous extension up to the boundary. Endowed with the norm $\norm{A}_r := \sup_{z \in \strip_r^d} \, \norm{A (z) }$, it is a Banach space.

In a previous work (see \cite{DK2}) we studied $\GL (m, \R)$-valued analytic cocycles. 
Here we will allow our cocycles to have {\em singularities} (i.e. points of non-invertibility), as long as they are not identically singular. This in particular ensures that all Lyapunov exponents are finite.

Let us then define  $\cocycles_m$ to be the set of cocycles $A \in \qpcmat{d}$ with $\det [ A (x) ] \not\equiv 0$. This condition implies in particular that all LE are finite. 

The set $\cocycles_m$ is open in $\qpcmat{d}$ and we let $ \dist (A, B) := \norm{A-B}_r$ be the induced distance on it. 
Then the collection $\cocycles := \{  (\cocycles_m,  \dist) \}_{m \ge 1}$ is the space of cocycles for our quasi-periodic model.

%


\medskip

For the reader's convenience we briefly recall some definitions and notations regarding the Lyapunov exponents, Oseledets filtrations and decompositions of a cocycle $A$ in any space of cocycles $\cocycles_m$.

The ergodic theorem of Kingman allows us to define the Lyapunov exponents $L_j(A)$ with $1\leq j\leq m$ as
$L_j(A):=\Lambda_j(A)-\Lambda_{j-1}(A)$ where
$$ \Lambda_j(A) := \lim_{n\to \infty} \frac{1}{n}\,\log \norm{\wedge_j A(x)} \quad \text{ for }\, \mu\text{-a.e. } x\in X\;. $$

Let $\tau=(1\leq \tau_1 <\ldots <\tau_k < m)$ be a signature. If  $A\in\cocycles_m$ has a $\tau$-gap pattern, i.e., $L_{\tau_j}(A)>L_{\tau_{j+1}}(A)$ for all $j$,
we define the Lyapunov $\tau$-block 
$$ \Lambda^\tau(A):= ( \Lambda_{\tau_1}(A),\ldots, \Lambda_{\tau_k}(A))\in\R^k\;.$$

A flag of $\R^m$ is any increasing sequence of linear subspaces. The corresponding sequence of dimensions is called its signature.
A measurable filtration is a measurable function
on $X$, taking values in the space of flags of $\R^m$ 
with almost sure  constant signature.
We denote by
$\Filt(X,\R^m)$ the space of measurable filtrations.
Note that the Oseledets filtration of $A$, which we denote by $\filt(A)$, is an element of this space.

We denote by  $\Filatp{\tau}(X,\R^m)$
the subset of measurable filtrations with a signature $\tau$ or finer. If $F\in \Filatp{\tau}(X,\R^m)$ there is a natural projection $F^\tau$ with signature $\tau$, obtained from $F$ by simply `forgetting' some of its components. This space is endowed with the following pseudo-metric
$$ \dist_\tau(F, F'):=
\int_X d_\tau( F^\tau(x), (F')^\tau(x) )\,  \mu(dx) \;,$$
where $d_\tau$ refers to the metric on the $\tau$-flag manifold.

On the space $\Filt(X,\R^m)$ we consider the coarsest topology that makes the sets $\Filatp{\tau}(X,\R^m)$ open, and the pseudo-metrics $\dist_\tau$ continuous.


A decomposition of  $\R^m$ is a sequence of linear subspaces $\{E_j\}_{1\leq j \leq k+1}$ 
  whose direct sum is $\R^m$. This determines the flag 
  $E_1 \subset E_1\oplus E_2\subset \ldots  \subset E_1\oplus\ldots \oplus E_k$, whose signature $\tau$ also designates the signature of the decomposition.

A measurable decomposition is a measurable function
on $X$, taking values in the space of decompositions of $\R^m$ 
with almost sure  constant signature.
We denote by
$\Dec(X,\R^m)$  the space of measurable decompositions.
Note that the Oseledets decomposition of $A$, which we denote by $\dec(A)$, is an element of this space.

We denote by   $\Decatp{\tau}(X,\R^m)$
the subset of measurable decompositions with a signature $\tau$ or finer. If $\dec \in  \Decatp{\tau}(X,\R^m)$ there is a natural restriction $\dec^\tau$ with signature $\tau$, obtained from $\dec$ by simply `patching up' the appropriate components. This space is endowed with the following pseudo-metric
$$ \dist_\tau(\dec, \dec'):=
\int_X d_\tau( \dec^\tau(x), (\dec')^\tau(x) )\,  \mu(dx) \;,$$
where $d_\tau$ refers to the metric on the manifold of $\tau$-decompositions.

On the space $\Dec(X,\R^m)$  we consider the coarsest topology that makes the sets $\Decatp{\tau}(X,\R^m)$ open, and the pseudo-metrics $\dist_\tau$ continuous.

\medskip

We are ready to state a general result on the continuity of the LE, the Oseledets filtration and the Oseledets decomposition for quasi-periodic cocycles.

\begin{theorem} \label{qp_ldt_thm:cont}
Assume that the translation $\om \in \rm{DC}_t$ for some $t >0$ and let $m\ge1$.

Then all Lyapunov exponents $L_j:\cocycles_{m}\to \R$, with $1\leq j \leq m$, 
the Oseledets filtration $\filt:\cocycles_{m}\to \Filt(X,\R^m)$,
and the Oseledets decomposition  $\dec:\cocycles_{m}\to \Dec(X,\R^m)$,
are continuous functions of the cocycle
$A\in \cocycles_{m}$. 

Moreover,  if $A\in \cocycles_{m}$ has a $\tau$-gap pattern then the functions
$\Lambda^\tau$, $\filt^\tau$ and $\dec^\tau$ are
weak-H\"older continuous in a neighborhood of $A$.
\end{theorem}

In the one-variable case $d=1$, and for translations that satisfy a slightly stronger (but still generic) Diophantine condition, we obtain a stronger modulus of continuity (see 
Section~\ref{refinements_1var}).

\subsection{Brief literature review}
\label{qp_lit_review}
 

In some sense the strongest result on continuity of the Lyapunov exponents for quasi-periodic cocycles for the {\em one-frequency} translation $d=1$ case, is due to A. \'Avila, S. Jitomirskaya and C. Sadel (see \cite{AJS}). The space of cocycles considered in this paper is the whole $C^{\om}_r (\T^1, \Mat (m, \C))$ (hence identically singular cocycles are not excluded) and the authors prove {\em joint continuity} in cocycle and frequency at all points $(A, \om)$ with $\om$ irrational. 
A previous work of S. Jitomirskaya and C. Marx in \cite{JitMarx-CMP} established a similar result for $\Mat (2, \C)$-valued cocycles, using a different approach. 
We note here that both approaches rely crucially on the {\em convexity} of the top Lyapunov exponent of the complexified cocycle, as a function of the imaginary variable, by establishing first  continuity {\em away} from the torus.  This method immediately breaks down in the {\em several} variables $d > 1$ case treated here.

The approaches of \cite{AJS} and \cite{JitMarx-CMP} are independent of any arithmetic constraints on the translation frequency $\om$ and they do not use large deviations. However, the results are not quantitative, in the sense that the they do not provide any modulus of continuity of the LE. All available quantitative results, from the classic result of M. Goldstein and W.Schlag (see \cite{GS-Holder}) to more recent results such as \cite{KTao-d, KTao-Jacobi, YouZhang-Holder-cont, DK2} or our current work, use some type of large deviations, whose derivation depends upon imposing appropriate arithmetic conditions on $\om$.

We note that in the (more particular) context of Schr\"odinger cocycles, joint continuity in the energy parameter and the frequency translation was proven for the one variable $d=1$ case by J. Bourgain and S. Jitomirskaya (see \cite{B-J}) and for the several variables $d > 1$ case by J. Bourgain (see \cite{B-d}). Both papers used weaker versions of large deviation estimates, proven under weak arithmetic (i.e. restricted Diophantine) conditions on $\om$, although eventually the results were made independent of any such restrictions.

Continuity properties of the Lyapunov exponents were also established for certain {\em non-analytic} quasi-periodic models (see \cite{sK1, sK2, WangZhang-cont-C2}).


Here we are dealing with both a base dynamics given by a translation on the {\em several} variables torus and a fiber action which has {\em singularities}.

Our approach is based in  an essential way upon establishing certain {\em uniform} estimates on analytic functions of several variables and on pluri subharmonic functions. 

The issue of {\em singularity} is especially delicate for {\em several} variables functions. One obstacle, for instance, is the fact that an analytic function of several variables may vanish identically along hyperplanes, while not being globally identically zero. Related to this, a pluri subharmonic function may be $- \infty$ along hyperplanes, while not being globally $- \infty$.

A crucial tool in our analysis is Theorem~\ref{L2:bound}, that shows that the obstacle described above for an analytic function can be removed with an appropriate change of coordinates. Another crucial tool in our analysis is the observation that for the pluri subharmonic functions corresponding to iterates of a cocycle, while they may have singularities as described above, these singularities can be captured by certain {\em analytic} functions.

Most of the work in this paper is devoted to proving a {\em uniform} LDT estimate for iterates of the cocycle. Uniform LDT estimates for cocycles with singularities were obtained before by S. Jitomirskaya and C. Marx (see \cite{JitMarx}) for $\Mat (2, \C)$-valued cocycles. Again, the approach in \cite{JitMarx} is one-variable specific.

Let us now phrase the LDT estimate obtained in this paper.

\begin{theorem}\label{thm:LDT-intro}
Given  $A \in \cocycles_m$ and $\om \in \rm{DC}_t$, there are constants $\delta = \delta (A) > 0$, $n_0 = n_0 (A, t) \in \N$, $\ldtmeas = \ldtmeas(A) < \infty$, $a = a (d) > 0$ and $b = b (d) > 0$  such that if $\normr{B-A} \le \delta$ and $n \ge n_0$ then  
\begin{equation} \label{eq:LDT-intro}
\abs{  \{ x \in \T^d \colon \abs{ \frac{1}{n} \log \norm{\Bn{n} (x) } - \LE{n}_{1} (B)  } > \ldtmeas \, n^{-a} \} } < e^{- n^b}
\end{equation}
\end{theorem}

Once~\eqref{eq:LDT-intro} is established, we only need to verify that we are in the context of the abstract continuity Theorem 3.1 in~\cite{LEbook}, see also Theorem 1.1  in~\cite{LEbook-chap3}. 

We note that the above LDT estimate is of independent interest. Such estimates have been widely used in the study of discrete quasi-periodic Schr\"odinger operators, to establish positivity of Lyapunov exponents, estimates on Green's functions, continuity of the integrated density of states and spectral properties (such as Anderson localization) for such operators (see J. Bourgain's monograph \cite{B}, see also \cite{sK1, sK2} and references therein).

The LDTs proven here, along with the other technical analytic tools, may then prove useful for future projects on topics in mathematical physics related to larger classes of discrete, quasiperiodic operators.

The rest of this paper  is organized as follows. In Section~\ref{psh_functions} we prove general uniform estimates on analytic and pluri subharmonic functions. These abstract results are then applied in Section~\ref{ldt_qp_proof} to quantities related to cocycles iterates, leading to the proof of the LDT. In Section~\ref{cont_le_qp} we explain how our system satisfies the assumptions of the general criterion in Chapter 3 of~\cite{LEbook} (see also~\cite{LEbook-chap3}). Finally, in Section~\ref{refinements_1var} we show that in the {\em one}-variable case, the LDT proven in Section~\ref{ldt_qp_proof}  can be used in an inductive argument that eventually leads to a {\em sharper} LDT, and in turn, to a {\em stronger} modulus of continuity for the Lyapunov exponents.

\section{Estimates on unbounded pluri-subharmonic functions}
\label{psh_functions}

In this section we derive certain uniform estimates on analytic functions of severable variables and on pluri-subharmonic functions. These estimates are of a general nature, and they will be applied in the next section to quantities related to iterates of analytic cocycles. Uniformity is understood relative to some measurements  which are stable under small perturbations of the functions being measured. The main technical difficulties in establishing these estimates are related to the non-trivial nature of the zeros of an analytic function of several variables, and correspondingly, to the unboundedness of the pluri-subharmonic functions we study.

\subsection{The uniform {\L}ojasiewicz inequality}\label{uLojasiewicz}
Throughout this paper, a quantitative description of quasi-analyticity, the {\L}ojasiewicz inequality, will play a crucial role. We make the observation that this property is \emph{uniform} in a small neighborhood of such a function.

\begin{lemma}\label{unif-loj}
Let $f (x) \in \analyticf{d}$ such that $f (x) \not \equiv 0$. Then there are constants $\delta = \delta (f) > 0$, $S = S (f) > 0$ and $b = b (f) > 0$ such that if $g (x) \in \analyticf{d}$ with $\normr{g - f} < \delta$ then
\begin{equation}\label{unif-loj-star}
\abs{ \{ x \in \T^d \colon \abs{g (x) } < t \} } < S \,t^b \quad \text{ for all } t > 0 \,.
\end{equation} 
\end{lemma} 

\begin{proof}
We may assume that $f (x)$ is not constant, otherwise $f (x) \equiv C$, $\abs{C} > 0$ and \eqref{unif-loj-star} is then obvious. 

{\L}ojasiewicz inequality~\eqref{unif-loj-star} for a \emph{fixed}, non-constant analytic function $f (x)$ of several  
variables has been established for instance in \cite{GS-Holder} (see Lemma 11.4), and for smooth, transversal functions in \cite{sK1} for $d=1$ (see Lemma 5.3) and in \cite{sK2} for $ d> 1$ (see Theorem 5.1). Moreover, the constants $S$ and $b$ in \cite{sK1, sK2} depend \emph{ explicitly} on some measurements of $f$, and it is easy to see that these measurements are uniform.

Assume for simplicity that $d=2$, although a similar argument holds for any $d \ge 1$.
 Then recall from \cite{sK2} that a smooth function $f (x)$ is called \emph{transversal} if for any point $x \in \T^2$ there is a multi-index $\alpha = (\alpha_1, \alpha_2) \in \N^2$, $\alpha \neq (0, 0)$ such that the corresponding partial derivative is non-zero: $\partial^{\alpha} f (x) \neq 0$. Clearly, non-constant analytic functions are smooth and transversal. 

By Lemma 5.1 in \cite{sK2}, for such a function $f$, there are $m = m (f) = (m_1, m_2) \in \N^2, m \neq (0, 0)$ and $c = c (f) > 0$ such that for any $x \in \T^2$ we have
\begin{equation}\label{unif-loj-eq1}
\abs{ \partial^{\alpha} f (x) } \ge c
\end{equation}
for some multi-index $\alpha = (\alpha_1, \alpha_2)$ with $\alpha_1 \le m_1, \alpha_2 \le m_2$.

Let 
\begin{equation}\label{unif-loj-eq2}
A (f) := \max \{ \abs{ \partial^{\alpha} f (x) }  \colon x \in \T^d, \alpha = (\alpha_1, \alpha_2), \alpha_1 \le m_1 + 1, \alpha_2 \le m_2 + 1 \} \,.
\end{equation}

Theorem 5.1 in \cite{sK2} says that
\begin{equation*}
\abs{ \{ x \in \T^2 \colon \abs{f (x) } < t \} } < S t^b \quad \text{ for all } t > 0
\end{equation*} 
where, according to the last line of its proof (see also Remark 5.1) $S = S(f) \sim A(f) \cdot m(f)$ and $b = b(f) = \frac{1}{3^{m(f)}}$.

Therefore, in order to obtain the uniform estimate~\eqref{unif-loj-star}, all we need to show is that the above constants $m = m (f), c = c(f), A = A(f)$ depend uniformly on the function $f$.

Indeed, let $g \in \analyticf{2}$ such that $\normr{g - f} < \delta$. By analyticity, for some constant $B=B(f)$ depending only on $\alpha$ and $r$, 
$$\normz{\partial^\alpha g - \partial^\alpha f} \le B  \,\delta \,,$$
hence if $\alpha = (\alpha_1, \alpha_2)$ with $\alpha_1 \le m_1 + 1, \alpha_2 \le m_2 +1$ and $m = (m_1, m_2) = m (f)$ from above, then
$$\normz{\partial^\alpha g - \partial^\alpha f} \le B (f) \delta = \frac{c(f)}{2}\,,$$
provided $\delta = \delta(f) := \frac{c(f)}{2 B(f)}$.

From~\eqref{unif-loj-eq1} we conclude that if $\normr{g-f} < \delta$, then for every $x \in \T^2$ 
\begin{equation*}
\abs{ \partial^{\alpha} g (x) } \ge \frac{c}{2}
\end{equation*}
holds for some multi-index $\alpha = (\alpha_1, \alpha_2)$ with $\alpha_1 \le m_1, \alpha_2 \le m_2$.

Moreover, for such functions $g$, the upper bound 
$A (g)$ satisfies
$$A (g) \le A(f) + \frac{c}{2} \sim A (f)\,,$$
which concludes the proof of the lemma. 
\end{proof}

\begin{lemma}\label{unif-loj-L2}
Let $f$ be a bounded function satisfying {\L}ojasiewicz inequality with constants $S, b$:
\begin{equation}\label{unif-loj-star2}
\abs{ \{ x \in \T^d \colon \abs{f (x) } < t \} } < S \, t^b \quad \text{ for all } t > 0 \,.
\end{equation}
Then
\begin{equation}\label{unif-loj-L2-star}
\norm{ \log \abs{f} }_{L^2 (\T^d)}   \le C \,,
\end{equation}
where $C$ is a finite  explicit constant depending only on $\norm{f}_{\infty}$, $S$ and $b$, i.e. on some measurements of $f$. 
\end{lemma}

\begin{proof}
The argument is straightforward. From \eqref{unif-loj-star2}, the set $\{ x \colon f (x) = 0 \}$ has zero measure. Split the rest of the phase space into 
\begin{align*}
E := & \{ x \colon \abs{f (x)} \ge 1\} \,,\\ 
E_n := & \{ x \colon \frac{1}{2^{n+1}} \le \abs{f (x)} \le \frac{1}{2^n} \}
\end{align*}
for all $n \ge 0$.

If $x \in E$, then $ \abs{ \log \abs{f(x)} } \le \abs{\log \norm{f}_{\infty}} < \infty$.

If $x \in E_n$, then $ \abs{ \log \abs{f(x)} } \less n+1$. 

Moreover, from \eqref{unif-loj-star2}, $\abs{ E_n} < S \, \left(\frac{1}{2^n}\right)^b = S \, \left(\frac{1}{2^b}\right)^n $.

Then 
\begin{align*}
\norm{ \log \abs{f} }^2_{L^2 (\T^d)}  & = \int_E \abs{ \log{\abs{f}} }^2 + \sum_{n=0}^{\infty} \int_{E_n} \abs{ \log{\abs{f}} }^2 \\
& \le  \abs{\log \norm{f}}^2_{\infty} +  S \, \sum_{n=0}^{\infty}  \frac{(n+1)^2}{(2^b)^n} \\
& =  \abs{\log \norm{f}}^2_{\infty} +  S \, \frac{ 2^{2b} (2^b + 1)}{(2^b-1)^3} \,.
\end{align*}
We then conclude:
$$
\norm{ \log \abs{f} }_{L^2 (\T^d)}  \less \abs{\log \norm{f}_{\infty}} + S (2^b -1)^{- 3/2} \,.  \qquad 
$$ 
\end{proof}

\begin{remark}\label{r-unif-loj-L2} \normalfont
The previous two lemmas imply that if $f \in \analyticf{d}$, $f \not \equiv 0$, then there are constants $\delta = \delta(f), C = C(f) < \infty$ such that 
$$
\norm{ \log \abs{g} }_{L^2 (\T^d)}   \le C 
$$
holds for any $g \in \analyticf{d}$ with $\normr{g-f} < \delta$.

\end{remark}

\smallskip

\subsection{Uniform $L^2$-bounds on analytic functions}\label{unif-l2bounds}

Consider the following norm on measurable functions
$f:\T^d\to\R$ 
$$ \normtwo{f} :=  \sup_{\substack{ x_1,\ldots, x_{j-1}, x_{j+1},\ldots, x_d\in\T\\
1\leq j\leq d } } \left(\int_\T \abs{f(x_1,\ldots, x_{j-1}, t,  x_{j+1},\ldots, x_d)}^2\, dt \right)^{1/2} \;.  $$
We say that a measurable function $f:\T^d\to\R$ is {\em uniformly separately  
$L^2$-bounded} \,  if \, $\normtwo{f}<+\infty$.

\begin{lemma}\label{triplenorm-transl-inv}
Given a translation $T$ on $\T^d$,
for any measurable function $f:\T^d\to\R$,
$\normtwo{f\circ T}=\normtwo{f}$.
\end{lemma}

\begin{proof}
Just use (in each variable) the translation invariance of the Lebesgue measure on $\T$. 
\end{proof}

\begin{definition}
We say that a function $f:\T^d\to \R$ {\em vanishes along an axis}
if there are $1\leq j\leq d$ and $x_1,\ldots, x_{j-1}, x_{j+1},\ldots, x_d\in\T$ so that
$f(x_1,\ldots, x_{j-1}, t, x_{j+1},\ldots, x_d)=0$ for every $t\in\T$.
\end{definition}

\begin{lemma} Given an analytic function $f:\T^d\to\R$,
if $f$ does not vanish along any axis then 
 $\log \abs{f}$ is uniformly separately  $L^2$-bounded.
\end{lemma}

\begin{proof}
The assumption implies that for all
$1\leq j\leq d$ and for all $x_1,\ldots, x_{j-1}$, 
$x_{j+1},\ldots, x_d\in\T$,
the analytic function
$$\varphi_{j; x_1,\ldots, x_{j-1}, x_{j+1},\ldots, x_d}(t):= f(x_1,\ldots, x_{j-1}, t, x_{j+1},\ldots, x_d)$$ 
is not identically zero.
Since clearly for all $1\leq j\leq d$, the set 
$$\{ \varphi_{j; x_1,\ldots, x_{j-1}, x_{j+1},\ldots, x_d}(t) \colon x_1,\ldots, x_{j-1}, x_{j+1},\ldots, x_d\in\T \} $$
is compact, applying Remark~\ref{r-unif-loj-L2} (with $d=1$) to the one-variable functions above, we conclude that there is $C = C(f) < \infty$ such that $\norm{\log\abs{ \varphi_{j; x_1,\ldots, x_{j-1}, x_{j+1},\ldots, x_d}}}_{L^2} < C$, which shows that  $\log \abs{f}$  is uniformly separately  $L^2$-bounded.

We note that in fact more can be shown, namely that  for all $1\leq j\leq d$, the function $H_j:\T^{d-1}\to \R$,
$H_j(x_1,\ldots, x_{j-1}, x_{j+1},\ldots, x_d):= \norm{\log\abs{ \varphi_{j; x_1,\ldots, x_{j-1}, x_{j+1},\ldots, x_d}}}_{L^2}$
is continuous, hence it has a maximum value, which leads to the same conclusion. 
\end{proof}

\begin{theorem}
\label{L2:bound}
For any analytic function $f\in\analyticf{d}$ with $f\not{\!\!\equiv} 0$,
there are $\varepsilon=\varepsilon(f,r)>0$, $C=C(f,r)>0$ and
there is a matrix $M\in\SL(d,\Z)$ such that for any
$g\in\analyticf{d}$ with $\norm{f-g}_r<\varepsilon$,
and for any $x_1,\ldots, x_{j-1}, x_{j+1}, \ldots, x_d\in\T$ with $1\leq j\leq d$,
$$ \norm{\log \abs{g \circ M (x_1,\ldots, x_{j-1},\, \cdot\, , x_{j+1},\ldots, x_d)} }_{L^2_{x_j}(\T)}\leq C\;.$$
In other words,  up to some
linear change of coordinates in the torus $\T^ d$, the functions
 $\log \abs{g}$ with $g$ near $f$ are uniformly separately  $L^2$-bounded.
\end{theorem}

\smallskip

Given $\delta>0$, define
$ \Sigma_\delta$ to be the set of all $k\in\Z^d$ such that 
any geodesic circle parallel to the vector through $k$  is  
$\delta$-dense in $\T^ d$.

%
%
%
%

%
%
%
%

\begin{proposition} Given $\delta>0$, there is a matrix $M\in\SL(d,\Z)$
such that every column of $M$ is in $\Sigma_\delta$.
\end{proposition} 

\begin{proof}
Take any matrix $M\in\SL(d,\Z)$ with non-negative entries which is primitive, and has a
characteristic polynomial $p_M(\lambda)=\det(M-\lambda I)$  irreducible over $\Z$ (see Lemma~\ref{example}).
Then $M$ has a dominant eigenvector $\omega\in{\rm int}(\R^d_+)$.
Consider the canonical projection $\pi:\R^d\to\T^d$ and
define $H=\overline{\pi (\langle \omega\rangle )}$ (topological closure in $\T^d$),
where $\langle \omega\rangle =\{\, t\,\omega\,:\, t\in\R\,\}$.
Define also $\mathfrak{h} = \pi^{-1}(H)$.
The set $H$ is a compact connected subgroup of $\T^ d$,
while $\mathfrak{h}$ is a linear subspace of $\R^ d$, the Lie algebra of $H$.
The group $H$ is invariant under the torus  automorphism  $\phi_M:\T^d\to \T^d$,
$\phi_M(x)=M x\, ({\rm mod}\,\Z^d)$, and hence (by restriction and quotient) the map $\phi_M$ induces the  tori automorphisms
$\phi_H \colon H\to H$ and $\widehat{\phi}_H:\T^ d/H\to \T^ d/H$.
Thus $M\,\mathfrak{h}=\mathfrak{h}$, and the linear maps of these tori automorphisms
at the level of Lie algebras are
$\Phi_\mathfrak{h}:\mathfrak{h}\to \mathfrak{h}$, $\Phi_\mathfrak{h}(x)=M\,x$,
and
$\widehat{\Phi}_\mathfrak{h}:\R^d/\mathfrak{h}\to \R^d/\mathfrak{h}$,
$\widehat{\Phi}_\mathfrak{h}(x+\mathfrak{h})=M\,x + \mathfrak{h}$.
The characteristic polynomials of these linear automorphisms have
integer coefficients because they are associated with tori automorphisms.
Finally, because $\mathfrak{h}$ is invariant under $M$, the characteristic polynomial
$p_M(\lambda)$ factors as the product of the characteristic polynomials
of $\Phi_\mathfrak{h}$ and $\widehat{\Phi}_\mathfrak{h}$.
Since the polynomial $p_M(\lambda)$ is irreducible,
we conclude that $H=\T^d$, which implies that the line spanned by $\omega$ is dense in $\T^d$.
Because $M$ is irreducible, the  lines spanned by the columns of $M^n$
approach the line  spanned by $\omega$ as $n\to+\infty$.
Hence for $n$ large enough, every column of $M^n$ lies in $\Sigma_\delta$. 
\end{proof}
Consider the following family of matrices in $\SL(d,\Z)$
\begin{equation}\label{irred:hyp:family}
 M_d= \left(\begin{array}{cccccc}
1 & 0 & 0 & \cdots & 0 & 1  \\ 
1 & 1 & 0 & \cdots & 0 & 1  \\
0 & 1 & 1 & \cdots & 0 & 0  \\
\vdots & \vdots & \vdots &  \ddots     & \vdots & \vdots  \\
0 & 0 & 0 &  \cdots     & 1 & 0  \\
0 & 0 & 0 &  \cdots     & 1 & 1  \\
\end{array}
\right)  \; \text{ if }\; d>2,\quad 
M_2= \left(\begin{array}{cc}
1 & 1 \\ 1 & 2
\end{array}
\right)
\end{equation}


\begin{lemma}\label{example}
The matrix $M_d$ is primitive and its characteristic polynomial is irreducible over $\Z$.
\end{lemma}

\begin{proof} Computing the characteristic polynomial
of the matrix $M_d$ with the Laplace determinant rule, we obtain
$$p_d(\lambda)=\det(M_d-\lambda I)= (-1)^d ((\lambda-1)^d-\lambda)\;.$$
In particular, $\det (M_d) = p_d(0)=(-1)^d(-1)^d = 1$.
Considering the permutation matrix
$$  P= \left(\begin{array}{cccccc}
0 & 0 & 0 & \cdots & 0 & 1   \\ 
1 & 0 & 0 & \cdots & 0 & 0   \\
0 & 1 & 0 & \cdots & 0 & 0  \\
\vdots & \vdots & \vdots &  \ddots     & \vdots &\vdots \\
0 & 0 & 0 &  \cdots   & 0 &   0  \\
0 & 0 & 0 &  \cdots   & 1 &   0  
\end{array}
\right)  $$
we have $M_d\geq I+P$, where the partial order $\geq $ refers to component-wise
comparison of the entries of the two matrices. Hence
$$(M_d)^d \geq (I+P)^d=\sum_{j=0}^d \binom{d}{j}\,P^j \;, $$
and since the right-hand-side has all entries positive, it follows that
$M_d$ is primitive.
%

Writing $\mu=\lambda-1$, we get
$(\lambda-1)^ d-\lambda = \mu^d-\mu-1$. Hence the irreducibility of
$p_d(\lambda)$ is equivalent  to that of $\mu^ d-\mu-1$, which was established
to hold for every $d\geq 2$ by Selmer (see Theorem 1 in~\cite{S56}). 
\end{proof}

\begin{proof}[Proof of Theorem~ \ref{L2:bound}]
Let $f:\T^d\to\R$ be an analytic function  such that $f\not{\!\!\equiv} 0$.
Take constants $c$ and $C$ such that $0<c<\norm{f}_\infty$ and $\norm{Df}_\infty<C<+\infty$.
Let $\varepsilon>0$ be such that for every $g\in\analyticf{d}$ with $\norm{g-f}_r < \delta$, 
one  still has $0<c<\norm{g}_\infty$ and $\norm{Dg}_\infty<C$.
Choose $\delta>0$ such that $C\,\delta<c$, and pick a matrix $M\in\SL(d,\Z)$
such that every column of $M$ lies in $\Sigma_\delta$.
Then any axis $\gamma$ of $\T^d$ along the coordinate system defined by $M$
has homotopy type in $\Sigma_\delta$. We cannot have $g\vert_\gamma\equiv 0$,
because the $\delta$-density of $\gamma$ implies the contradiction
$c<\norm{g}_\infty \leq C\,\delta <c$. Therefore, $g\circ M$ does not 
vanish along any axis, which implies that
 $\log \abs{g\circ M}$ is uniformly separately  $L^2$-bounded.
 
 As before, by a compactness argument, Remark~\ref{r-unif-loj-L2} (with $d=1$) implies that for some constant $C = C(f) < \infty$, the inequality
  $$ \norm{\log \abs{g(x_1,\ldots, x_{j-1},\, \cdot\, , x_{j+1},\ldots, x_d)} }_{L^2_{x_j}(\T)} < C $$
holds for all $(x_1,\ldots, x_{j-1}, x_{j+1},\ldots, x_d)\in\T^{d-1}$
 and $g\in \analyticf{d}$ near $f$.

Therefore, $\normtwo{ \log \abs{g} }$ is
  uniformly bounded in a neighborhood of $f$. 
\end{proof}

\begin{remark} \label{to-change-or-not-to-change-vars} \normalfont
There is an alternative argument that does not require a change of coordinates, but which is technically much more complicated. It involves proving, using the uniform {\L}ojasiewicz inequality~\eqref{unif-loj-star}, that for any $g \in \analyticf{d}$ which is close enough to an $f \in \analyticf{d}$ with $f \not\equiv 0$, and for any $n \gg 1$, 
$$ \norm{\log \abs{g(x_1,\ldots, x_{j-1},\, \cdot\, , x_{j+1},\ldots, x_d)} }_{L^2_{x_j}(\T)}\le n$$
holds for all $(x_1,\ldots, x_{j-1},\,  x_{j+1},\ldots, x_d) \in \T^{d-1}$ outside of a set of measure $ < e^{- c n^{1/2}}$.

In other words, without a change of variables, while the uniform $L^2$-bound in Theorem~\ref{L2:bound} may not hold,  we still get a polynomial bound off of an exponentially small set of inputs and in each variable.  When applying these estimates to quantities related to cocycles iterates,  the number $n$ above will be correlated with the number of iterates, and this non-uniform but {\em polynomial} $L^2$-bound will be manageable. 

\end{remark}

\begin{remark} \label{Loj-d-var-new proof} \normalfont
Besides being a very helpful technical tool in this paper, the fact that after a change of coordinates, a two variables analytic function which is not identically zero, will not vanish identically along any horizontal or vertical  line, could be of independent use elsewhere. 

We note here, for instance,  that it can be used to derive the {\L}ojasiewicz inequality in two variables from the one-variable statement. That is because after a change of coordinates, the one-variable statement is applicable along {\em any} horizontal line, hence via a compactness argument and Fubini, the two-variables statement follows. 

Of course, once one has a {\em uniform} {\L}ojasiewicz inequality in one variable (see \cite{sK1} or \cite{JitMarx}), this argument also provides a uniform statement in two-variables. 

The several (instead of two) variables situation is similar.

\end{remark}

\smallskip

\subsection{Estimates on unbounded subharmonic functions}\label{estimates_sh}
In this subsection we review some crucial estimates on \index{Subharmonic function} subharmonic functions. That is, we list the {\em one variable} tools to be employed later in the derivation of the base LDT estimates for pluri-subharmonic observables.

We begin by reminding the reader some elementary facts about subharmonic functions.
She may consult the monographs \cite{HK-book} or  \cite{Levin} for the classical theory. 

\begin{definition}\label{sh:def}
Let $\Omega$ be a domain of $\C$. A function $u \colon \Omega \to [ - \infty, \infty)$ is called subharmonic in $\Omega$ if for every $z \in \Omega$, $u$ is upper semicontinuous at $z$ and it satisfies the sub-mean value property: 
$$u (z) \le  \int_{0}^{1} u (z + r e (\theta)) d \theta\,,$$
for some $r_0 (z) > 0$ and for all $r \le r_0 (z)$.
\end{definition}

Basic examples of subharmonic functions are $\log \sabs{z - z_0}$ or more generally, $\log \sabs{f (z)}$ for some analytic function $f (z)$ or $\int \log \sabs{z - \zeta} d \mu (\zeta)$ for some positive measure with compact support in $\C$.

The maximum of a finite collection of subharmonic functions is subharmonic, while the supremum of a collection (not necessarily finite) of subharmonic functions is subharmonic provided it is upper semicontinuous. 
In particular this implies that  if $A \colon \Omega \to \Mat (m, \C)$ is a matrix valued analytic function, then
$$u (z) := \log \norm{ A (z)} = \sup_{\norm{v}, \norm{w} \le 1} \, \log \abs{\avg{A (z) \, v, w}}$$
is subharmonic in $\Omega$. 

Note that the function $u (z)$ defined above is bounded from above on any compact subdomain. However, unless say $A \colon \Omega \to \GL (m, \C)$, the subharmonic function $u (z) = \log \norm{ A (z)}$ may be unbounded from below or it may assume the value $- \infty$.

A fundamental result in the theory of subharmonic functions is the Riesz representation theorem, which we formulate below.

\begin{theorem}\label{Riesz-rep-thm}
Let $u (z)$ be a subharmonic function in a domain $\Omega$, and assume that $ u (z) \not\equiv - \infty$. Then there is a unique Borel measure $\mu$ on $\Omega$ called the Riesz measure of $u$, such that for every compactly contained subdomain $\Omega_1$,
$$u (z) = \int_{\Omega_1} \log \sabs{z - \zeta} d \mu (\zeta) + h (z)\,,$$
where $h (z)$ is a harmonic function on $\Omega_1$.
\end{theorem}

Let us assume for now that the subharmonic function $u(z)$ is bounded: 
$$\sabs{u(z)} \le C \quad \text{ for all } z\in\Omega\,.$$ 

Using Jensen's formula for subharmonic functions (see Section 7.2 in \cite{Levin}), for every $z \in \Omega$ and $r > 0$, if the disk  $D (z, r) \subset \Omega$, then we have
$$ \int_0^r \frac{\mu (D (z, t))}{t} d t = \int_0^1 u (z + r e (\theta)) d \theta  - u (z)\,.$$

Since clearly
$$\int_0^r \frac{\mu (D (z, t))}{t} d t \ge \int_{r/2}^r \frac{\mu (D (z, t))}{t} d t \more \mu (D (z, r/2))\,,$$
we conclude that 
\begin{equation}\label{measurement-sh-eq0}
\mu (D (z, r/2)) \less C\,.
\end{equation}

Therefore, we obtain the following measurement on the total Riesz mass of $u$:
$$\mu (\Omega_1) \less C (\Omega, \Omega_1) \, C\,,$$
where $C$ is the bound on $u (z)$ and $C (\Omega, \Omega_1)$ is a constant that depends on how the subdomain $\Omega_1$ is covered by disks contained in $\Omega$.

An argument that uses the Poisson-Jensen representation formula (see Section 3.7 in \cite{HK-book}) leads to a similar bound on the $L^\infty$-norm (on a slightly smaller compactly contained subdomain $\Omega_2$) for the harmonic part $h (z)$ in the Riesz representation of $u (z)$. 
We conclude that if $ \sabs{u (z)} \le C$ on $\Omega$, then
\begin{equation}\label{measurement-sh-bounded}
\mu (\Omega_1) +  \norm{h}_{L^{\infty} (\Omega_2)} \le C (\Omega, \Omega_1, \Omega_2) \, C\,.
\end{equation} 

We refer to the quantity in \eqref{measurement-sh-bounded} as a {\em uniform measurement} on the {\em bounded} subharmonic function $u (z)$, since it only depends on its bound and on its domain. It is precisely this measurement that determines the parameters in the base-LDT estimates for the observable $u(x)$.

This paper requires similar estimates for subharmonic functions that are {\em unbounded} from below. A uniform measurement like \eqref{measurement-sh-bounded} on the Riesz mass and on the harmonic part of such a subharmonic function were obtained by M. Goldstein and W. Schlag (see Lemma 2.2 in \cite{GS-fineIDS}). The derivation is based on the Poisson-Jensen formula and on considerations that involve Green's functions. The result in  \cite{GS-fineIDS} is formulated for functions $u \colon \Omega \to \R$. It holds, however, also for $ u \colon \Omega \to [- \infty, \infty)$, as long as $u \not\equiv - \infty$. That is because  $u \not\equiv - \infty$, along with some assumptions on the boundary of $\Omega$, are the only requirements for the applicability of the Poisson-Jensen formula (see Section 3.7 in \cite{HK-book}).

We formulate the aforementioned result in \cite{GS-fineIDS} for a subharmonic function on an annulus, as this is the context of our model.

\begin{lemma}\label{measurement-sh:lemma}
Let $u \colon \strip_r \to [-\infty, \infty)$ be a subharmonic function, and let
$$u (z) = \int_{\strip_{r/2}} \log \sabs{z - \zeta} d \mu (\zeta) + h (z)$$
be its Riesz representation on the smaller annulus $\strip_{r/2}$.
Assume that
\begin{equation}\label{measurement-sh:hyp1}
\sup_{z \in \strip_r} u (z) - \sup_{z \in \strip_{r/2}} u (z) \le C
\end{equation}

Then
\begin{equation}\label{measurement-sh:eq}
\mu (\strip_{r/2}) +  \norm{h}_{L^{\infty} (\strip_{r/4})} \le C_r \, C\,,
\end{equation}
where $C_r$ is a constant that depends on the width $r$ of the annulus.
\end{lemma}

\begin{remark}\label{measurement-sh:remark1} \normalfont
Lemma~\ref{measurement-sh:lemma} above says that in order to obtain a uniform measurement like \eqref{measurement-sh-bounded} for a subharmonic function $u(z)$, it is enough to have an upper bound everywhere and a lower bound at {\em some}  point. Clearly assumption \eqref{measurement-sh:hyp1} is implied (up to doubling the constant) by
\begin{equation}\label{measurement-sh:hyp2}
\sup_{z \in \strip_r} u (z) + \norml{u}{\Lp{2}}  \le C\,.
\end{equation}
\end{remark}

\begin{remark}\label{measurement-sh:remark2} \normalfont We comment on the order of magnitude of the constant $C_r \, C$ in \eqref{measurement-sh:eq}, as $r \to 0$.

In the bounded case $\sabs{u (z} \le C$ for all $z \in \strip_r$, from \eqref{measurement-sh-eq0} and the fact that $\strip_{r/2}$ can be covered by $\mathscr{O} (\frac{1}{r})$ many disks of radius $\mathscr{O} (r)$, it follows that the total Riesz mass of $u$ is of order $\frac{1}{r} \, C$ or less. The $L^\infty$ bound on $h$ will be of the same order, showing that $C_r \, C \less \frac{1}{r} \, C$.  

In the unbounded case, under the assumption \eqref{measurement-sh:hyp1}, an inspection of the proof of Lemma 2.2 in \cite{GS-fineIDS} shows that the constant $C_r$ in \eqref{measurement-sh:eq} depends only on the annulus $\strip_r$, via certain estimates on its Green's function and an argument involving Harnack's inequality. These considerations lead to an estimate on $C_r$ which is  {\em exponential} in $\frac{1}{r}$ as $r \to 0$. One can show, via some calculations involving elliptic integrals, that this estimate on the order of $C_r$ cannot be significantly improved, unless of course \eqref{measurement-sh:hyp1} is strengthened.

We note, however, that throughout this paper, the width $r$ of the annulus $\strip_r$ is {\em fixed}. We do perform a change of coordinates of the multivariable torus $\T^d$, which in turn affects the size of the domain of the relevant  subharmonic functions. However, this change of coordinates is performed {\em only once}. Hence for all intents and purposes, the constant $C_r$ in this paper may be treated as a universal constant, and so the uniform measurement on $u(z)$ given by \eqref{measurement-sh:eq} depends only on the bound in \eqref{measurement-sh:hyp1} or in \eqref{measurement-sh:hyp2}.
\end{remark}

\medskip

We formulate the crucial estimates on a subharmonic function $u(z)$ which are needed in the proof of the LDT: a rate of decay of the Fourier coefficients of $u (x)$ and an estimate on its BMO norm derived under an appropriate splitting assumption. The reader may consult \cite{Muscalu-S} or \cite{Duoa} for background on the relevant harmonic analysis topics. 

Let $u \colon \strip_r \to [-\infty, \infty)$ be a subharmonic function, and let
$$u (z) = \int_{\strip_{r/2}} \log \sabs{z - \zeta} d \mu (\zeta) + h (z)$$
be its Riesz representation on the smaller annulus $\strip_{r/2}$.

Assume that
\begin{equation}\label{measurement-sh}
\mu (\strip_{r/2}) +  \norm{h}_{L^{\infty} (\strip_{r/4})}  +  \norml{\partial_x h}{\Lp{\infty}}  \le \umeas\,.
\end{equation}

\begin{lemma}\label{decay-FC-1var} 
Under the assumptions above, the following estimate on the Fourier coefficients of $u (x)$ as a function on $\T$ holds:
\begin{equation}\label{sl-decay-FC-1var}
\abs{ \hat{u} (k) } \less \umeas \, \frac{1}{\abs{k}} \quad \text{ for all } k \in \Z, \, k \neq 0.
\end{equation}
\end{lemma}

\begin{lemma}\label{decay-BMO-1var} Let $u (z)$ be a subharmonic function satisfying \eqref{measurement-sh}.
Assume moreover that  there is a splitting $u = u_0 + u_1$  with $\norml{u_0}{\Lp{\infty}} < \ep_0$ and $\norml{u_1}{\Lp{1}} < \ep_1$. Then $u (x)$ has the following BMO bound:
 \begin{equation}\label{sl-BMO-1var}
\norml{u}{BMO (\T)} \less \ep_0 + ( \umeas \, \ep_1 )^{1/2}\,. 
\end{equation}
\end{lemma}

\smallskip

We make some comments regarding the proofs of these lemmas. 

These types of results are available for {\em bounded} subharmonic functions, see Chapter 4 in J. Bourgain's monograph \cite{B} or Section 1 in \cite{B-d}.  

A careful inspection of their proofs shows that boundedness of the subharmonic function $u(z)$ is not strictly necessary: it is only used to derive estimates on the Riesz mass $\mu (\strip_{r/2})$ and on the $L^\infty$-norm of the harmonic part $h$ in the Riesz representation of $u$, that is, to derive the uniform measurement \eqref{measurement-sh-bounded} on the annulus. 

Moreover, the resulting constants appearing in the estimates on the Fourier coefficients and on the BMO norm of $u(x)$ depend precisely on this uniform measurement on $u$ and on the bound on the derivative of $h$, that is, on $\norml{\partial_x h}{\Lp{\infty}}$. 

Therefore, the bound in \eqref{measurement-sh} can be substituted for the boundedness of $u(z)$ and Lemmas~\ref{decay-FC-1var} and \ref{decay-BMO-1var} are proven along the same lines as their counterparts in \cite{B, B-d}.

\begin{remark}\label{decays-sh:remark} \normalfont
Let $u \colon \strip_r \to [-\infty, \infty)$ be a subharmonic function and assume that \eqref{measurement-sh:hyp2} holds, that is
\begin{equation}\label{measurement-sh:hyp}
\sup_{z \in \strip_r} u (z) + \norml{u}{\Lp{2}}  \le C\,.
\end{equation}

Then from Lemma~\ref{measurement-sh:lemma} we have
$$\mu (\strip_{r/2}) +  \norm{h}_{L^{\infty} (\strip_{r/4})} \le C_r \, C\,.$$

Using the Poisson integral representation for harmonic functions and scaling, it is easy to see that since $ \norm{h}_{L^{\infty} (\strip_{r/4})} \le C_r \, C$, then 
$ \norml{\partial_x h}{\Lp{\infty}} \less \frac{C_r }{r} \, C$.

We conclude that 
$$\mu (\strip_{r/2}) +  \norm{h}_{L^{\infty} (\strip_{r/4})}  +  \norml{\partial_x h}{\Lp{\infty}}  \less \frac{C_r }{r} \, C\,,$$
or in other words, the assumption \eqref{measurement-sh} above holds with $\umeas = \mathscr{O} \bigl( \frac{C_r }{r} \, C \bigr)$.

We also note that since the annulus $\strip_r$ will be fixed throughout the paper, the uniform measurement in \eqref{measurement-sh} will depend only on the bound $C$ in \eqref{measurement-sh:hyp}. That is, we may write
$$\mu (\strip_{r/2}) +  \norm{h}_{L^{\infty} (\strip_{r/4})}  +  \norml{\partial_x h}{\Lp{\infty}}  \less C\,.$$
\end{remark}

The following result is an immediate consequence of Lemma~\ref{decay-BMO-1var} 
and it shows that a a weak a-priori concentration inequality for a subharmonic function can always be {\em boosted}  to a stronger estimate. We use the notation $\avg{u}$ to denote the space average $\int_\T u(x) d x$ of the function $u (x)$ on $\T$.

\begin{lemma}\label{splitting-lemma-1var}
Let $u (z)$ be a subharmonic function such that the bound  \eqref{measurement-sh} holds.
If $u (x)$ satisfies the weak a-priori estimate:
\begin{equation}\label{sl-1var-weak}
\abs{ \{ x \in \T \colon \abs{ u (x) - \avg{u} }  > \ep_0 \} } < \ep_1
\end{equation} 
with $\ep_1 \le \ep_0^4$, then for some absolute constant $c > 0$,
\begin{equation}\label{sl-1var-strong}
\abs{ \{ x \in \T \colon \abs{ u (x) - \avg{u} }  > \ep_0^{1/2} \} } < e^{- c \bigl[\ep_0^{1/2} + \umeas \, \ep_1^{1/4} \, \ep_0^{-1/2}  \bigr] ^{-1}} < e^{- c \, \umeas^{-1} \, \ep_0^{-1/2}}.
\end{equation} 
\end{lemma}

\begin{proof}
We may of course assume that $\avg{u} = 0$. 
Denote the set in \eqref{sl-1var-weak} by $\B$, so $\abs{\B} < \ep_1$.
Then $u = u_0 + u_1$ for $u_0 := u \cdot \ind_{\B^{\comp}}$ and $u_1 := u \cdot \ind_\B$.

Clearly $\norml{u_0}{\Lp{\infty}} \le \ep_0$.
Since \eqref{measurement-sh} holds, we may apply Lemma~\ref{decay-FC-1var} and \eqref{sl-decay-FC-1var} implies $\norml{u}{\Lp{2}} \less \umeas$, so
$\norml{u_1}{\Lp{1}} \le \norml{u}{\Lp{2}} \cdot \norml{\ind_\B}{\Lp{2}} \le \umeas \, \ep_1^{1/2}$.

\smallskip

Lemma~\ref{decay-BMO-1var} applies, and we have the BMO bound:
$$\norml{u}{BMO(\T)} \le \ep_0 +  \umeas \, \ep_1^{1/4}.$$

The conclusion then follows directly from John-Nirenberg inequality.
\end{proof}

\subsection{Base LDT estimates for pluri-subharmonic observables}
\label{uqBET}
The Birkhoff ergodic theorem implies that if the translation by $\om$ is ergodic, then for any observable $\xi \in L^1 (\T^d)$, as $n \to \infty$ the Birkhoff average 
$$\frac{1}{n} \sum_{j=0}^{n-1} \xi (x + j \om) \to \avg{\xi} \quad \text{ for a.e. } x \in \T^d.$$ 
Moreover, if $\xi$ is continuous, then the convergence above is uniform in the phase $x$.

In order to establish {\em fiber}-LDT estimates, we need a {\em quantitative} version of this convergence, one which applies to observables that admit a pluri-subharmonic extension. Moreover, the parameters describing this quantitative convergence should depend explicitly on a certain {\em uniform measurement} of such observables.

We formulate and prove this quantitative version of the Birkhoff ergodic theorem for pluri-subharmonic functions that are {\em unbounded} from below but otherwise satisfy some bounds on average. 

A similar result for {\em bounded} pluri-subharmonic functions was formulated and proven in \cite{sK2} (see also \cite{B, B-d} for results in the same spirit).  

\smallskip

We begin with some general considerations on pluri-subharmonic functions.

\begin{definition}
Let $\Omega$ be a domain in  $\C^d$. A function $u \colon \Omega \to [ - \infty, \infty)$ is called pluri-subharmonic if it is upper semicontinuous and its restriction to any complex line is subharmonic. 
\end{definition}

It follows from the definition above that the composition of a pluri-subharmonic function with a linear function is pluri-subharmonic as well. Moreover, a pluri-subharmonic function is subharmonic in each variable.

If $f (z)$ is holomorphic (i.e. analytic in each variable), then $u (z) = \log \sabs{f (z)}$ is pluri-subharmonic. Moreover, if $A(z)$ is a holomorphic matrix-valued function, then $u (z) = \log \norm{A(z)}$ is pluri-subharmonic. 

\smallskip

We note two important differences between subharmonic and pluri-subharmonic functions.

Firstly, the zeros set of an analytic function of {\em one} variable is discrete, which is of course not the case for several variables analytic functions. Correspondingly, if the function $u (z)$ is pluri-subharmonic, then the set $\{ z \colon u (z) = - \infty\}$ can be quite complex, i.e. a variety of co-dimension $1$ in $\C^d$. In particular, this set may contain hyperplanes or lines parallel to the Euclidian coordinate axes. 

Secondly, the Riesz representation theorem~\ref{Riesz-rep-thm}, which is an important tool in the study of subharmonic functions, is not available for pluri-subharmonic functions.

\smallskip

Crucial to proving a quantitative Birkhoff ergodic theorem for observables on $\T$ with subharmonic extension to $\strip_r$, is having the decay~\ref{sl-decay-FC-1var} on its Fourier coefficients from Lemma~\ref{decay-FC-1var} and the boosting of a concentration inequality in Lemma~\ref{splitting-lemma-1var}.

Similar results for a pluri-subharmonic function $u (z)$ on $\strip_r^d$ can be obtained through a slicing argument, provided we may apply Lemma~\ref{decay-FC-1var} and Lemma~\ref{splitting-lemma-1var} in {\em each} variable, and with the same measurements. 

More precisely, this type of argument works if the measurement  \eqref{measurement-sh} applies {\em uniformly} for all subharmonic functions 
$$\strip_r \ni z_i \to u (z_1, \ldots, z_{i-1}, z_i, z_{i+1}, \ldots, z_d)\,,$$ where $1 \le i \le d$ and $z_1, \ldots, z_{i-1}, z_{i+1}, \ldots, z_d \in \strip_r$. 

Of course this would be automatic if the pluri-subharmonic function $u (z)$ were bounded. A weaker assumption is for \eqref{measurement-sh:hyp2} to hold uniformly in each variable. In other words, we require that $u(z)$ be bounded from above on $\strip_r^d$ and that its $L^2$-norm in {\em each variable} be bounded as well. The latter assumption is equivalent to having a bound on $\normtwo{u}$ (see Subsection~\ref{unif-l2bounds} for the meaning of this norm).  

To summarize, the assumption we make in this section on a pluri-subharmonic function $u \colon \strip_r^d \to [-\infty, \infty)$ is that for some constant $C < \infty$ we have
\begin{equation}\label{psh-hyp}
\sup_{z \in \strip_r^d} u (z) + \normtwo{u} \le C\,.
\end{equation}

We now state the analogue of the boosting Lemma~\ref{splitting-lemma-1var} for several variables. For simplicity we consider $d=2$ variables, but a similar result, proven the same way, holds for any number $d$ of variables. The meaning of the constant $C_r$ below was given in Subsection~\ref{estimates_sh} (see Remark~\ref{measurement-sh:remark2}). We remind the reader that $C_r$ depends only on the annulus $\strip_r$, hence only on $r$. 

\begin{lemma}\label{splitting-lemma-2var}
Let $u (z)$ be a pluri-subharmonic function on $\strip_r^2$ such that the bound \eqref{psh-hyp} holds for some constant $C < \infty$ and let $\umeas := \frac{C_r}{r} \, C$.

If $u (x)$ satisfies the weak a-priori estimate:
\begin{equation}\label{sl-2var-weak}
\abs{ \{ x \in \T^2 \colon \abs{ u (x) - \avg{u} }  > \ep_0 \} } < \ep_1
\end{equation} 
with $\ep_1 \le \ep_0^8$, then for some absolute constant $c > 0$,
\begin{equation}\label{sl-2var-strong}
\abs{ \{ x \in \T^2  \colon \abs{ u (x) - \avg{u} }  > \ep_0^{1/4} \} } < ^{- c \bigl[\ep_0^{1/4} + \umeas \, \ep_1^{1/8} \, \ep_0^{-1/2}  \bigr] ^{-1}} < e^{- c \, \umeas^{-1} \, \ep_0^{-1/4}}.
\end{equation} 

For $d$ variables, replace the powers $1/4$ by $1/2 d$ etc. 
\end{lemma}

A similar result was proven in \cite{B} (see Lemma 4.12) for {\em bounded} pluri-subharmonic functions, using a slicing argument and the corresponding one variable result. 

The reader may verify that the argument in \cite{B} applies as long as the one variable result, i.e. Lemma~\ref{splitting-lemma-1var}, can be applied uniformly to the functions $u (\cdot, z_2)$, $u(z_1, \cdot)$, for all $z_1, z_2 \in \strip_r$,  and provided these functions are also uniformly bounded in say $L^2 (\T)$. These conditions are of course ensured by the assumption \eqref{psh-hyp}.

\smallskip

We are now ready to formulate the main result of this section, a quantitative Birkhoff ergodic theorem.

\begin{theorem}\label{qBET-thm}
Let $u \colon \strip_r^d \to [-\infty, \infty)$ be a pluri-subharmonic function satisfying
\begin{equation}\label{psh-hyp-qBET}
\sup_{z \in \strip_r^d} u (z) + \normtwo{u} \le C\,.
\end{equation}

Let $\om \in \rm{DC}_t$ and put $n_0 := t^{-2}$. There is $a = a (d) > 0$ so that for all $n \ge n_0$
\begin{equation}\label{qBET-star-prop}
\abs{ \{ x \in \T^d \colon \abs{ \frac{1}{n} \sum_{j=0}^{n-1} u (x + j \om) - \avg{u} } > \umeas \, n^{- a} \} } < e^{- c \, n^a},
\end{equation}
where $\umeas = \mathscr{O} \bigl( \frac{C_r }{r} \, C \bigr)$ and $c = \mathscr{O} (1)$.
\end{theorem}

\begin{proof} We prove this statement for $d=2$ variables, but the same argument holds for any number of variables. We follow the same strategy used in the proof of Proposition 4.1 in \cite{sK2}, as the main ingredients of the argument depend only on having a uniform decay  on the  Fourier coefficients of $u$ (separately in each variable) and on the boosting Lemma,~\ref{splitting-lemma-2var}, both of which are ensured by the assumption \eqref{psh-hyp-qBET}. For the reader's convenience, we present the complete argument here. 

The assumption \eqref{psh-hyp-qBET} implies that the bound \eqref{measurement-sh:hyp} holds in each variable and with the same constant $C$, i.e. for all subharmonic functions $u (\cdot, z_2)$, $u(z_1, \cdot)$ with $z_1, z_2 \in \strip_r$. Hence by Remark~\ref{decays-sh:remark}, the bound \eqref{measurement-sh} with $\umeas = \mathscr{O} \bigl( \frac{C_r }{r} \, C \bigr)$ holds for all these functions as well. 

This ensures that we can apply 
Lemma~\ref{decay-FC-1var} on the decay of the Fourier coefficients in each variable and obtain
 \begin{equation}\label{Riesz} 
\sup_{x_2 \in \T} \, \abs{\hat{u} (l_1, x_2)}  \le \umeas \cdot \frac{1}{\sabs{l_1}}  \ \text{ and  } \ 
\sup_{x_1 \in \T} \, \abs{\hat{u} (x_1, l_2)}  \le \umeas \cdot \frac{1}{\sabs{l_2}} 
\end{equation}
 for all $l = (l_1, l_2) \in \Z^2$ with $l_1 \neq 0, l_2 \neq 0$.
 
Expand $u (x) = u \x$ into a Fourier series
$$u \x = \avg{u} + \sum_{\genfrac{}{}{0cm}{}{\li \in \Z^2}{\li \neq (0,0)}} \, \hat{u} \li \cdot e (\li \cdot \x)\,.$$

Then the Birkhoff averages have the form
\begin{align*}
& \frac{1}{n} \sum_{j = 0}^{n-1}  u (\x + j \omm)  \\
& =   \avg{u} +  \sum_{\genfrac{}{}{0cm}{}{\li \in \Z^2}{\li \neq (0,0)}} \, \hat{u} \li \cdot e (\li \cdot \x) \cdot \Bigl( \frac{1}{n} \sum_{j = 0}^{n-1}  e ( j \, \li \cdot \omm) \Bigr )  \\
& =   \avg{u} +  \sum_{\genfrac{}{}{0cm}{}{\li \in \Z^2}{\li \neq (0,0)}} \, \hat{u} \li \cdot e (\li \cdot \x) \cdot  K_n (\li \cdot \omm)\,, 
\end{align*}
where we denoted by $K_n (y)$ the Fej\'{e}r kernel on $\T$:
$$K_n (y) = \frac{1}{n} \sum_{j = 0}^{n-1}  e ( j \, y) \, = \, \frac{1}{n} \, \frac{1 - e (n y)}{1 - e (y)}\,.$$

Clearly $K_n (y)$ has the following bound:
\begin{equation}\label{fejerkernelbound}
\abs{K_n (y)} \le \min \Bigl\{ 1, \frac{1}{n \norm{y}} \Bigr\}\, 
\end{equation}
where $\norm{y}$ was defined in~\eqref{DC}.

We then have:
\begin{align*}
& \bnorm{ \frac{1}{n} \sum_{j = 0}^{n-1}  u (\x + j \omm) - \avg{u} }_{L^2(\T^2)}^2  \\
&  = \sum_{\genfrac{}{}{0cm}{}{\li \in \Z^2}{\li \neq (0,0)}} \, \abs{\hat{u} \li}^2 \cdot \abs{K_n (\li \cdot \omm)}^2 \\
& =  \sum_{1 \le \sabs{l_1} + \sabs{l_2} <  K} \, \abs{\hat{u} \li}^2 \cdot \abs{K_n (\li \cdot \omm)}^2 \\
& + \sum_{\sabs{l_1} + \sabs{l_2}  \ge K} \, \abs{\hat{u} \li}^2 \cdot \abs{K_n (\li \cdot \omm)}^2\,.
\end{align*}
We will estimate the second sum above using the bounds \eqref{Riesz} on the Fourier coefficients of $u \x$ and the first sum using the Diophantine condition on the frequency $\omm$. The splitting point $K$ will be chosen to optimize the sum of these estimates.

Clearly \eqref{Riesz} implies:
$$\sum_{l_2 \in \Z} \, \abs{\hat{u} \lli}^2 = \norm{\hat{u} (l_1, x_2)}_{L_{x_2}^2(\T)}^2 \le \, \Bigr( \umeas \, \frac{1}{\sabs{l_1}}\Bigl)^2 \le  \, \umeas^2 \, \frac{1}{\sabs{l_1}^2}\,,$$
and
$$\sum_{l_1 \in \Z} \, \abs{\hat{u} \lli}^2 = \norm{\hat{u} (x_1, l_2)}_{L_{x_1}^2(\T)}^2 \le \, \Bigr( \umeas \, \frac{1}{\sabs{l_2}}\Bigl)^2 \le  \, \umeas^2 \, \frac{1}{\sabs{l_2}^2}\,.$$

Then we have:
\begin{align*}
 & \sum_{\sabs{l_1} + \sabs{l_2}  \ge K} \, \abs{\hat{u} \li}^2 \cdot \abs{K_n (\li \cdot \omm)}^2   \le \sum_{\sabs{l_1} + \sabs{l_2}  \ge K} \, \abs{\hat{u} \li}^2 \\
  &  \qquad \le  \sum_{\li \colon \sabs{l_1} \ge K/2} \, \abs{\hat{u} \li}^2  +    \sum_{\li \colon \sabs{l_2} \ge K/2}  \abs{\hat{u} \li}^2    \less  \,  \umeas^2 \, \frac{1}{K}\,. \\
 \end{align*}

Estimate \eqref{Riesz} clearly impies:
$$\abs{\hat{u} \li} \le \umeas \, \frac{1}{\sabs{l_1} + \sabs{l_2}}\,.$$

Then using the Diophantine condition on $\omm$ and \eqref{fejerkernelbound}, we obtain:
\begin{align*}
& \sum_{1 \le \sabs{l_1} + \sabs{l_2}  <  K} \, \abs{\hat{u} \li}^2 \cdot \abs{K_n (\li \cdot \omm)}^2 \\
& \qquad \le  \umeas^2 \, \sum_{1 \le \sabs{l_1} + \sabs{l_2}  <  K} \,\frac{1}{(\sabs{l_1} + \sabs{l_2})^2} \cdot \frac{1}{n^2 \, \norm{\li \cdot \omm}^2} \\
& \qquad \le  \umeas^2 \, \sum_{1 \le \sabs{l_1} + \sabs{l_2}  <  K} \,\frac{1}{(\sabs{l_1} + \sabs{l_2})^2} \cdot \frac{(\sabs{l_1} + \sabs{l_2})^{2 (d+1)}}{n^2 \, t^2} \\
& \qquad  \less   \umeas^2 \,  \frac{K^{2 (d+1)}}{n^2 \, t^2}\,.
\end{align*}

We conclude:
$$\bnorm{ \frac{1}{n} \sum_{j = 0}^{n-1}  u (\x + j \omm) - \avg{u} }_{L^2(\T^2)} \ \le \ \umeas \, \Bigr( \frac{1}{K^{1/2}} + \frac{K^{(d+1)}}{n \, t} \Bigl) \le \umeas \, n^{- a}$$
for some positive constant $a$ that depends on $d$ and for $n \ge 1/ t^2$.

Applying Chebyshev's inequality, we have:
\begin{equation}\label{apriori-avshifts}
\abs{ \{ x \in \T^2 : \abs{ \frac{1}{n} \sum_{j = 0}^{n-1}  u (x + j \om)  \, -  \avg{u} }  > \umeas \, n^{- a/5} \} } \ < \  
 n^{- 8 a/5}\,.  
\end{equation}

This estimate is of course too weak, since the size of the exceptional set is only polynomially small. We will boost it by using Lemma~\ref{splitting-lemma-2var}. 

Consider the average 
$$v (z) := \frac{1}{\umeas} \, \frac{1}{n} \sum_{j=0}^{n-1} u (z + j \om)\,.$$

Then $v$ is a pluri-subharmonic function on $\strip_r^2$ and it clearly satisfies the bounds
$\displaystyle \sup_{z \in \strip_r^2} v (z)  \le \frac{1}{\umeas} \, C \le 1$ and  $\normtwo{ v }  \le \frac{1}{\umeas} \,  \normtwo{u} \le 1$,
hence
$$\sup_{z \in \strip_r^2} v (z)  + \normtwo{ v }  \less 1\,.$$

Since $\avg{ v} = \frac{1}{\umeas} \avg{u}$, from \eqref{apriori-avshifts} we have:
$$\abs{ \{ x \in \T^2 : \abs{ v (x) -  \avg{ v} } > \, n^{- a/5} \} } \ < \  n^{- 8 a/5}\,.  $$

Let $\ep_0 := n^{- a/5}$, $\ep_1 := n^{- 8 a/5} = \ep_0^8$. 

Apply Lemma~\ref{splitting-lemma-2var} to the function $v (z)$ to obtain:
$$\abs{ \{ x \in \T^2 : \abs{ v (x) -  \avg{ v} } > \, n^{- a/20} \} } \ < \  e^{- c \, n^{a/20}}\,,$$
where $c = \mathscr{O} (1)$. 
This completes the proof. 
\end{proof}

The assumption \eqref{psh-hyp-qBET} in Theorem~\ref{qBET-thm} will not necessarily be satisfied by the pluri-subharmonic functions associated to our model. That is because they are of the form $u(z) = \log \norm{A (z)}$, for some matrix valued analytic function $A(z)$, and $\norm{A(z)}$ may very well vanish identically along some lines parallel to coordinate axes. However, they will satisfy a weaker property that will allow us to handle their singularities via a change of coordinates, and replace them with pluri-subharmonic functions for which  \eqref{psh-hyp-qBET} does hold. We describe this idea in the following proposition.

\begin{proposition}\label{qBET-main}
Let $f \in \analyticf{d}$, $f \not \equiv 0$, $\om \in \rm{D C}_t$ and $C > 0$. There are constants $\delta = \delta (f, r) > 0$,  $a = a (d) > 0$, $k_0 = k_0 (f, r, d) \in \N$ and $\umeas = \umeas (f, r, C) < \infty$ such that if $u \colon \strip_r^d \to [-\infty, \infty)$ is a pluri-subharmonic function satisfying the bounds
\begin{equation}\label{ulb-qBET}
- C + \frac{1}{m} \sum_{j=0}^{m-1} \log \abs{ g (z + j \om) } \le u (z) \le C \quad \text{ for all } z \in \strip_r^d,
\end{equation}
for some $g \in \analyticf{d}$ with $\normr{g-f} < \delta$ and for some $m \ge 1$, then   
\begin{equation}\label{qBET-star}
\abs{ \{ x \in \T^d \colon \abs{ \frac{1}{n} \sum_{j=0}^{n-1} u (x + j \om) - \avg{u} } > \umeas \,n^{- a} \} } < e^{-  n^a}
\end{equation}
holds for all  $n \ge n_0 := t^{-2} \, k_0$.
\end{proposition}

\begin{proof}
By Theorem~\ref{L2:bound}, since $f \not \equiv 0$, there is $M \in \SL (d, \Z)$ such that in the new coordinates $x' = M x$, the analytic function $f (x')$ does not vanish identically along any hyperplane, and in fact, $\log \sabs{ f(x')}$ has the property that its $L^2$-norms separately in each variable are uniformly bounded, or in other words that $\normtwo{ \log \sabs{f (x')}}$ is bounded.  Moreover, this holds uniformly in a neighborhood of $f$. 

We note that the conclusion \eqref{qBET-star} of the theorem is coordinate agnostic.  
Indeed, if $M \in \SL (d, \Z)$ then $M^{- 1}$ preserves the measure while if $\om \in \rm{D C}_t$ then $\om' = M^{- 1} \om  \in \rm{D C}_{t'}$ for $t' = t / \norm{M}_1^{d+1}$, hence the Diophantine condition is preserved up to a constant. 

Furthermore, since $M$ is linear, $u \circ M (z)$ is also pluri-subharmonic, and its domain contains $\strip_{r'}$, where $r' \sim r/\norm{M}$. That is because the linear map induced by $M$ expands the imaginary direction, hence the width of the domain of $u \circ M (z)$ is proportionally smaller. However, $\norm{M}$ is a fixed constant depending upon $M$, hence only upon $f$. 

Hence it is enough to prove   \eqref{qBET-star} for $u (x)$ replaced by $u \circ M (x)$ and for $\om$ replaced by $M^{- 1} \om$.

Moreover, the assumption \eqref{ulb-qBET} holds for $u \circ M (x)$ provided we replace $g (x)$ by $g \circ M (x)$, which is still $\delta$-close to $f \circ M (x)$. 

\smallskip

Therefore, with no loss of generality, we may assume that for $\delta = \delta (f, r)$ small enough, and for some finite constant $C_0 = C_0 (f, r)$ we have: if $g \in \analyticf{d}$ with $\normr{g-f} < \delta$, then
\begin{equation}\label{qbet-1}
\normtwo{ \log \sabs{g}}  \le C_0\,.
\end{equation}

From \eqref{ulb-qBET} we have:
$$
\abs{ u(x) } \le 2 C + \frac{1}{m} \sum_{j=0}^{m-1} \abs{ \log \abs{ g (x + j \om) } } \quad \text{for all } x \in \T^d.
$$

Then using Lemma~\ref{triplenorm-transl-inv},
$$
\normtwo{u} \le 2 C + \normtwo{ \log \sabs{g}} \le 2 C + C_0,
$$
hence
$$
\sup_{z \in \strip_r^d} u (z) + \normtwo{u} \le C\,.
$$

This shows that the assumption \eqref{psh-hyp-qBET} in Theorem~\ref{qBET-thm} is now (after a change of coordinates) satisfied

We apply Theorem~\ref{qBET-thm} to conclude that \eqref{qBET-star} holds with $\umeas = \mathscr{O} \bigl(  \frac{C_r}{r} \, ( 3C + C_0) \bigr) = \umeas (f, r, C)$ and for all $n \ge n_0$, where $n_0 :=  t^{-2} \, k_0$.
We choose $k_0 = k_0 (f, r, d)$ such that $k_0 \ge  \norm{M}_1^{2 (d+1)}$ and such that the constant $c = \mathscr{O} (1)$ in  \eqref{qBET-star-prop} is absorbed.
\end{proof}

\section{The proof of the fiber LDT estimate}
\label{ldt_qp_proof}
Given any cocycle $A \in \cocycles_m$, we derive some uniform measurements on $A$ which will allow us to apply the base LDT estimate in Proposition~\ref{qBET-main} to all the iterates of any cocycle $B$ near $A$, in a {\em uniform} way. This, combined with an almost invariance under the base dynamics property for the iterates of the cocycle, will lead to the proof of the uniform fiber LDT estimate.

\subsection{Uniform measurements on the cocycle}\label{unif-meas}
We introduce some notations. For a cocycle $A \in \qpcmat{d}$, let 
$$f_A (z) := \det [A (z) ] \,.$$ 
Then clearly $f_A \in \analyticf{d}$. 

Moreover, for every scale $n \ge 1$, let
$$
 \un{n}{A} (z) :=  \frac{1}{n} \log \norm{\An{n} (z)} \,.
 $$
 Note that due to the analyticity of the cocycle $A (z)$, the functions $\un{n}{A} (z)$ are {\em pluri-subharmonic} on $\strip_r^d$. This property is crucial in establishing the fiber LDT estimate.
 
 We denote the space averages of these functions by
 $$
 \LE{n}_1 (A)  = \int_{\T^d}  \un{n}{A} (x) \, d x = \int_{\T^d}  \frac{1}{n} \log \norm{\An{n} (x)} \, d x\,. 
 $$

 The following proposition introduces some locally uniform measurements on a cocycle. It shows that the above functions are bounded in the $L^2$ - norm, uniformly in the scale, and uniformly in a neighborhood of a given non identically singular cocycle. It also shows that the failure of the above functions to be bounded in the $L^\infty$-norm is captured by Birkhoff averages of a one dimensional cocycle.
 
 \begin{proposition} \label{uplowbounds}
 Given a cocycle $A \in \qpcmat{d}$ with $f_A = \det [A] \not \equiv 0$, there are constants $\delta = \delta (A) > 0$ and $C = C (A) < \infty$, such that for any cocycle  $B \in \qpcmat{d}$, if $\normr{B-A} < \delta$, then
 \begin{align}
 & f_B = \det [B]   \not \equiv 0\,, \label{ulbounds-1} \\
 & \norm{ \log \abs{f_B} }_{L^2 (\T^d)}   \le C \label{ulbounds-2}
 \end{align}
 and for all $n \ge 1$  we have 
 \begin{align}
 & - C + \frac{1}{n} \sum_{i=0}^{n-1} \log \abs{f_B (\transl^i z)}   \le  \un{n}{B} (z)   \le C \,, \label{ulbounds-3} \\
 & \quad  \norm{ \un{n}{B}}_{L^2 (\T^d)}   \le C\,. \label{ulbounds-4}
 \end{align}
 \end{proposition}

 \begin{proof}
 Clearly the map $\qpcmat{d} \ni B \mapsto \det[B] \in \analyticf{d}$ is locally Lipschitz, hence we can choose $\delta = \delta (A) > 0$ small enough such that if $\norm{B-A}_r < \delta$, then the analytic function $g := f_B$ satisfies the {\L}ojasiewicz inequality ~\eqref{unif-loj-star} with the same constants $S = S(f_A)$, $b=b(f_A)$ as $f := f_A$ (see Lemma~\ref{unif-loj}). In particular, $f_B \not \equiv 0$ and from Remark~\ref{r-unif-loj-L2} we have the uniform $L^2$ - bound:
 $$  \norm{ \log \abs{f_B} }_{L^2 (\T^d)}   \le C (\norm{f_A}_{\infty}, S, b) \le C_2 (A)\,.$$
 
The upper bound in \eqref{ulbounds-3} is clear: if $\normr{B-A} < \delta$, then $\normr{B} \sim \normr{A}$, and since for every $z \in \strip_r^d$ we have
 $$\norm{ \Bn{n} (z) } \le \prod_{i=0}^{n-1} \norm{ B (\transl^i z) }
\le \normr{B}^n \,,$$

we conclude that 
$$\un{n}{B} (z) = \frac{1}{n} \log \norm{\Bn{n} (z) } \le \log \normr{B} < C_1(A)\,.$$

To establish the lower bound, we use Cramer's rule:
$$\det [B (z)] \cdot I = B (z)  \cdot {\rm adj} (B (z)) \;. $$ 
Hence
\begin{align*}
\prod_{i=0}^{n-1} f_B (T^i z) \cdot I  & = \prod_{i=0}^{n-1} \det [B (T^i z)] \cdot I\\
& = B (T^{n-1} z)   \ldots   B (z) \cdot  {\rm adj} (B (z))  \ldots  {\rm adj} (B (T^{n-1} z))\\
& = \Bn{n} (z) \cdot {\rm adj} (B (z))  \ldots  {\rm adj} (B (T^{n-1} z))\;.
\end{align*}
Clearly
$$\norm{ {\rm adj} (B (z))} \less \norm{B (z)}^{m-1} \le \normr{B}^{m-1} \quad \text{ for all } z\;.$$
This implies, for some $C_1 = C_1(A) \sim \abs{\log \normr{A}}$,
$$
\un{n}{B} (z) = \frac{1}{n} \log \norm{\Bn{n} (z) }  \ge - C_1 +  \frac{1}{n} \sum_{i=0}^{n-1} \log \abs{f_B (\transl^i z)}\,,
$$
which establishes \eqref{ulbounds-3}. 
Moreover, \eqref{ulbounds-3} also implies that for all $x \in \T^d$,
$$\abs{ \un{n}{B} (x) } \le 2 C_1 + \frac{1}{n} \sum_{i=0}^{n-1} \abs{ \log \abs{f_B (\transl^i x)} }\;. $$
Hence by \eqref{ulbounds-2} and the measure invariance of the translation $\transl$,
$$\norm{ \un{n}{B}}_{L^2 (\T^d)}   \le 2 C_1 +  \norm{ \log \abs{f_B} }_{L^2 (\T^d)}   \less C_1 + C_2\;. \qquad $$
 \end{proof}

\subsection{The nearly almost invariance property}\label{nAI}
For $\GLmR$-valued cocycles, the functions $\un{n}{A} (x)$ are almost invariant under the base dynamics $\transl$, in the sense that for \emph{all} $x \in \T^d$,
$$\abs{ \un{n}{A} (x) - \un{n}{A} (\transl x) } \le C \frac{1}{n}$$

For a  cocycle that has singularities but it is not identically singular, we establish this property off of an exponentially small set of phases. It is in fact crucial to obtain a \emph{uniform} in the cocycle bound on the measure of this exceptional set of phases. 

\begin{proposition}\label{unif-a-inv}
Let $A \in \qpcmat{d}$ such that $\det [A (x) ] \not \equiv 0$. Then there are constants $\delta = \delta(A) > 0$ and $C = C(A) < \infty$, such that for any $a \in (0,1)$, if $ B \in \qpcmat{d}$ with $\normr{B-A} < \delta$, then
\begin{equation}\label{unif-a-inv-star}
\abs{ \un{n}{B} (x) - \un{n}{B} (\transl x) } \le C \frac{1}{n^{a}}
\end{equation}
holds for all $n \ge 1$ and for all $x \notin  \B_n$, where $\abs{\B_n} < e^{- n^{1-a}}$.

The exceptional set $\B_n$ may depend on the coycle $B$, but its measure does not. 
\end{proposition} 

\begin{proof}
For any cocycle $B \in \qpcmat{d}$, let $f_B (x) := \det [B(x)] \in \analyticf{d}$. Then if $\normr{B-A} < \delta$ we have $\normr{f_B - f_A} < C \, \delta$, where $C = C(A) > 0$.

Using Lemma~\ref{unif-loj}, there are constants $\delta, C, b > 0$, all depending only on $A$, such that if $\normr{B-A} < \delta$ then
\begin{equation}\label{unif-a-inv-eq1}
\abs{ \{ x \in \T^d \colon \abs{f_B (x) } < t \} } < C t^b \quad \text{ for all } t > 0 \,.
\end{equation} 
In particular (or  by Fubini), the set $\Zm_0 (B) := \{ x\in \T^d \colon f_B (x) = 0 \}$ has zero measure, and so does $\cup_{n\ge0} \, \transl^{-n} \Zm_0 (B) =: \Zm (B)$.

Hence if $x \notin \Zm (B)$, then $B (\transl^n x)$ is invertible for all $n \ge 0$ and we can write:
\begin{align*}
& \frac{1}{n} \log \norm{\Bn{n} (x)} - \frac{1}{n} \log \norm{\Bn{n} (\transl x)}   \\
& \qquad = \frac{1}{n} \log \frac{ \norm{ B (\transl^n x)^{-1} \cdot [ B (\transl^n x) \cdot B (\transl^{n-1} x) \cdot \ldots \cdot B (\transl x)    ] \cdot B (x)  } }{ \norm{   B (\transl^n x) \cdot B (\transl^{n-1} x) \cdot \ldots \cdot B (\transl x)   } }   \\
& \qquad \le \frac{1}{n} \log [ \norm{ B (\transl^n x)^{-1} } \cdot \norm{B (x)} ] \le \frac{C}{n} + \frac{1}{n} \, \log \norm{B (\transl^n x)^{-1}  }\,.
\end{align*}
The last inequality follows from 
$$\norm{B (x)} \le \normr{B-A} + \normr{A} < \delta + \normr{A} \sim \normr{A}\,,$$
thus $\log \norm{B(x)} \le C(A)$ for all $x \in \T^d$.

Similarly,
\begin{align*}
& \frac{1}{n} \log \norm{\Bn{n} (\transl x)} - \frac{1}{n} \log \norm{\Bn{n} (x)}  \\
&\qquad  = \frac{1}{n} \log \frac{ \norm{ B (\transl^n x) \cdot [  B (\transl^{n-1} x) \cdot \ldots \cdot B (\transl x) \cdot B(x) ] \cdot B (x)^{-1} } }{ \norm{   B (\transl^{n-1} x) \cdot \ldots \cdot B (x)   } } \\
& \qquad \le \frac{1}{n} \log [ \norm{ B (\transl^n x) } \cdot \norm{B (x)^{-1}} ] \le \frac{C}{n} + \frac{1}{n} \, \log \norm{B (x)^{-1}  } \,.
\end{align*}

We conclude that if $x \notin \Zm(B)$, then
\begin{align}
& \abs{ \frac{1}{n} \log \norm{\Bn{n} (x)} - \frac{1}{n} \log \norm{\Bn{n} (\transl x)}  }  < \notag \\
&\qquad \qquad  \frac{C}{n} + \frac{1}{n} \, \log \norm{B (\transl^n x)^{-1}  } + \frac{1}{n} \, \log \norm{B (x)^{-1}  } \;. \label{unif-a-inv-eq2}
\end{align}

Therefore, in to prove~\eqref{unif-a-inv-star} we need to obtain a uniform in $B$ upper bound for $\norm{B (x)^{-1}}$, where $x$ is outside an exponentially small set. 

Upper bounds on the norm of the inverse of a matrix are obtained from lower bounds on the determinant via Cramer's rule. 
Indeed, if $x \notin \Zm(B)$ then
$$\norm{B (x)^{-1}} = \norm{{\rm adj} (B (x))} \cdot \frac{1}{\abs{\det [B(x)]}} \le C_1 \cdot \frac{1}{\abs{f_B (x)}}\,,
$$
which holds because
$$  \norm{{\rm adj} (B (x))} \less \norm{B(x)}^{m-1} \le \normr{B}^{m-1} \le C_1(A)\,.$$ 

Fix $a \in (0,1)$ and apply~\eqref{unif-a-inv-eq1} with 
$$t := \left(\frac{1}{C}\right)^{1/b} \, e^{- 1/b\, n^{1-a}} \,.$$

Then there is a set $\B_n \subset \T^d$ with $\abs{\B_n} < C t^b = e^{- n^{1-a}}$, such that if $ x \notin \B_n$ then $\abs{f_B (x)} \ge t$. 
Thus
 $$\norm{B (x)^{-1}} \le C_1 \frac{1}{t} = C_1 \, C^{1/b} \, e^{ 1/b\, n^{1-a}} = C_2 (A) \, e^{1/b\, n^{1-a}} \,,$$
 hence
 $$\frac{1}{n} \log \norm{B (x)^{-1}} \le  \frac{\log C_2}{n}  + 1/b \, \frac{n^{1-a}}{n} = C_3 (A) \, \frac{1}{n^{a}}$$
 which, combined with ~\eqref{unif-a-inv-eq2}, proves the proposition. 
\end{proof}

\subsection{The statement and proof of the LDT}\label{LDTsing-proof}
\begin{theorem}\label{LDTsing-thm}
Given  $A \in \qpcmat{d}$ with $\det [A(x) ] \not \equiv 0$ and $\om \in \rm{DC}_t$, there are constants $\delta = \delta (A) > 0$, $k_0 = k_0 (A) \in \N$, $ \ldtmeas= \ldtmeas(A, r) < \infty$, $a = a (d) > 0$ and $b = b (d) > 0$  such that if $\normr{B-A} \le \delta$ and $n \ge n_0 := t^{-2} \, k_0$, then 
\begin{equation} \label{LDTsing-star}
\abs{  \{ x \in \T^d \colon \abs{ \frac{1}{n} \log \norm{\Bn{n} (x)} - \LE{n}  (B)  } > \ldtmeas \, n^{-a} \} } < e^{- n^b}.
\end{equation}
 \end{theorem}
 
 \begin{proof} 
  
 Using Proposition~\ref{uplowbounds}, there are constants $\delta = \delta(A)$, $C = C (A)$, such that if $B$ is a cocycle near $A$: $\normr{B-A} < \delta$, then for all scales $n \ge 1$
  $$- C + \frac{1}{n} \sum_{i=0}^{n-1} \log \abs{f_B (\transl^i z)} \le  \un{n}{B} (z)   \le C\,.$$

We apply Proposition~\ref{qBET-main} with $f = f_A$, $C = C (A)$, so the dependence of the constants on the data will be: $\delta = \delta (f_A, r) = \delta (A)$, $a = a (d)$, $k_0 = k_0 (f_A, r, d) = k_0 (A)$ and $\umeas = \umeas (f_A, r, C) = \umeas (A)$.
  
  Since the map $\qpcmat{d} \ni B \mapsto f_B \in \analyticf{d}$ is locally Lipschitz, by possibly decreasing $\delta$, we may assume that whenever $\normr{B-A} < \delta$ we have that $\normr{f_B - f_A}$ is small enough that the pluri-subharmonic function $u (z) = \un{n}{B} (z)$ satisfies the assumption \eqref{ulb-qBET} with $g = f_B$. Hence \eqref{qBET-star} applied to our context says that for all $R \ge n_0$ we have:
  \begin{equation}\label{LDTsing-p1}
  \abs{ \{ x \in \T^d \colon \abs{ \frac{1}{R} \sum_{j=0}^{R -1} \un{n}{B}  (x + j \om) - \avg{\un{n}{B}} } > \umeas \, R^{- a} \} } < e^{- c R^a}
  \end{equation}  
 
 From the nearly almost invariance property given by Proposition~\ref{unif-a-inv}, after possibly decreasing $\delta$ and for a constant $C' = C' (A) < \infty$, we have  that if $\normr{B-A} < \delta$, then
 \begin{equation*}
\abs{ \un{n}{B} (x) - \un{n}{B} (\transl x) } \le C' \frac{1}{n^{a}}
\end{equation*}
for all $n \ge 1$ and for all $x \notin  \B_n$, where $\abs{\B_n} < e^{- n^{1-a}}$.

 Let $\Bbar_n := \cup_{i=0}^{R-1} \, \transl^{-i} \B_n$. Hence $\abs{\Bbar_n} \le R \, e^{- n^{1-a}}$ and if $x \notin \Bbar_n$ then
\begin{equation}\label{LDTsing-p2}
\abs{ \un{n}{B} (x) -  \frac{1}{R} \sum_{j=0}^{R -1} \un{n}{B}  (x + j \om)  } \le C' \frac{R}{n^{a}}
\end{equation}

Pick $R \ll n^a$, say $R = \intpart{n^{a/(a+1)}}$ and let $\ldtmeas = 2 (C' + \umeas)$. The conclusion \eqref{LDTsing-star} of the theorem then follows from \eqref{LDTsing-p1} and \eqref{LDTsing-p2} for some easily computable choice of the new parameter $a$ and the parameter $b$. 
 
 \end{proof}
 
 \begin{remark} \normalfont
 What determines {\em all} constants in the LDT estimate above, are precisely some measurements on the function $f_A (x) := \det [ A (x) ]$ and the parameter $t$ of the Diophantine condition $\rm{DC}_t$ on the frequency $\om$. 
 
 We also note that unlike in the case of random cocycles, the fiber-LDT estimate above was proven without assuming the existence of a gap between the first two Lyapunov exponents. In particular, using exterior powers, we can derive an LDT estimate for every Lyapunov exponent.
\end{remark}

\section{Deriving continuity of the Lyapunov exponents}
\label{cont_le_qp}
To prove the continuity of the LE, we use the abstract criterion given by Theorem 3.1 in Chapter 3 of our monograph~\cite{LEbook} (or Theorem 1.1 in our preprint~\cite{LEbook-chap3}).  Under the same assumptions, in Chapter 4 of~\cite{LEbook} (see also our preprint~\cite{LEbook-chap4}) we obtained abstract criteria for the continuity of the Oseledets filtration (Theorem 4.7 in~\cite{LEbook} or Theorem 3.2 in~\cite{LEbook-chap4}) and of  the Oseledets decomposition (Theorem 4.8 in~\cite{LEbook} or Theorem 3.3 in~\cite{LEbook-chap4}).

For the reader's convenience, we briefly review the relevant definitions.
We then explain how the aforementioned continuity results are applicable to the context of this chapter.

\begin{proof}\textit{(of Theorem~\ref{qp_ldt_thm:cont})}
The torus $\T^d$ together with the $\sigma$-algebra of Borel sets, the Haar measure and the translation by a rationally independent vector $\om$ form an ergodic system. 
Let $\cocycles = \bigcup_{m\ge 1} \, \cocycles_m$ be the space of analytic, not identically singular cocycles over this ergodic system defined in Section~\ref{qp_model}.

We say that a cocycle $A$ is uniformly $L^2$-bounded if there are constants $C = C(A) < \infty$ and $\delta = \delta (A) > 0$ such that
$$
\bnorm{ \frac{1}{n} \, \log \norm{\Bn{n} (\cdot)} }_{L^2 (\T^d)} < C
$$
for all cocycles $B$ that are close enough to $A$ (in the given topology on the space of cocycles)  and for all scales $n \ge 1$.

Estimate~\eqref{ulbounds-4} in proposition~\ref{uplowbounds} shows that all cocycles in $\cocycles$  are uniformly $L^2$-bounded.

Given a cocycle $A \in \cocycles_m$ and $N \in \N$, note that the sets $\{ x \in \T^d \colon \norm{\An{n} (x) } \le c \}$ or $\{ x \in \T^d \colon \norm{\An{n} (x) } \ge c \}$
for some $1 \le n \le N$ and $c > 0$ are closed Jordan measurable, so the algebra $\mathscr{F}_N (A)$ generated by them consists only of Jordan measurable sets. 

We say that a set $\Xi$ of observables and a cocycle $A \in \cocycles$ are compatible, if for any $N \in \N$, $F \in \mathscr{F}_N (A)$, $\ep > 0$, there is $\xi \in \Xi$ such that $\ind_F \le \xi$ and $\int_{\T^d} \xi \, d x \le \abs{F} + \ep$.

Let $\Xi := \mathscr{C}_0 (\T^d)$ be the set of all {\em continuous} observables $\xi \colon \T^d \to \R$. 
By the  regularity of the Borel measure, there is an open set $U \supseteq \overline{F}$ such that $\abs{U} \le \abs{F} + \ep$.
By Urysohn's lemma, there is a continuous function $\xi \in \Xi$ such that $0 \le \xi \le 1$, $\xi \equiv 1$ on $\overline{F}$ and $\xi \equiv 0$ on $U\comp$. Then
$$\ind_F \le \xi \quad \text{ and }\  \int_{\T^d} \xi \, d x \le \abs{U} \le \abs{F} + \ep\,,$$
which shows that $\Xi$ is compatible with every cocycle in $\cocycles$, a property we call the compatibility condition.

We call deviation size function any non-increasing map $\devf \colon (0, \infty) \to (0, \infty)$, and deviation measure function any sufficiently fast decreasing function $\mesf \colon (0, \infty) \to (0, \infty)$. A triplet $(\nzerobar, \devf, \mesf)$, where $\nzerobar \in \N$ and $\devf, \mesf$ are deviation size  or measure functions is called an LDT parameter, while any set $\params$ containing such triplets is called a set of LDT parameters.

We say than an observable $\xi \in \Xi$ satisfies a base-LDT estimate with parameter space  $\params$, if for every $\ep > 0$ there is an LDT parameter $(\nzerobar, \devf, \mesf)\in\params$ such that for all $n \ge \nzerobar$ we have $\devf (n) \le \ep$ and 
$$
\abs{ \{ x \in \T^d \colon \frac{1}{n} \, \sum_{j=0}^{n-1} \xi (\transl^j x) - \int_X \xi \, d x > \devf (n) \}  } < \mesf (n)\;.
$$

A torus translation by a rationally independent vector $\om$ is {\em uniquely ergodic}, hence the convergence in Birkhoff's ergodic theorem is {\em uniform} for {\em continuous} observables.
This shows that the base-LDT estimate holds trivially for $\xi \in \Xi$, with $\devf \equiv \ep$ and $\mesf \equiv 0$. 

Finally, a cocycle $A \in \cocycles_m$ is said to satisfy a uniform fiber-LDT with parameter space $\params$, if for every $\ep > 0$ there are $\delta > 0$ and an LDT parameter $(\nzerobar, \devf, \mesf)\in\params$ which may only depend upon $A$ and $\ep$, such that for any $B \in \cocycles_m$ with $\dist (B, A) < \delta$ and for all $n \ge \nzerobar$, we have $\devf (n) \le \ep$ and
$$
\abs{ \{ x \in \T^d \colon \abs{ \frac{1}{n} \, \log \norm{ \Bn{n} (x) } - \LE{n}_1 (B) } > \devf (n)  \}  }< \mesf (n)\;.
$$

Theorem~\ref{LDTsing-thm} shows that a uniform fiber-LDT estimate holds 
 for all cocycles in $\cocycles$,
with the parameter space $\params$
being the set of all triplets $(\nzerobar, \devf, \mesf)$ with $\nzerobar \in \N, \, \devf (t) \equiv \ldtmeas \, t^{- a}$ and  $\mesf (t) \equiv e^{ t^b}$ where the constants $\nzerobar, a, b$ are explicitly described in terms of some measurements of $A$ and the Diophantine condition on $\om$.

Theorem 1.1 in~\cite{LEbook-chap3} says that given an ergodic system, a space of cocycles, a set of observables and a set of LDT parameters, if the compatibility condition, the uniform $L^2$-boundedness, the base-LDT and the uniform fiber-LDT estimates hold, then all Lyapunov exponents are continuous. Moreover, if $A \in \cocycles_m$ has a Lyapunov spectrum gap, i.e. for some $1 \le k \le m$, $L_k (A) > L_{k+1} (A)$,  then locally near $A$, the map $B \mapsto \Lambda_k (B) = (L_1 + \ldots + L_k) (B)$ has a modulus of continuity 
$\omega (h) :=  [ \mesf \ (c \, \log \frac{1}{h}) ]^{1/2}$, for some $c > 0$ and some deviation measure function $\mesf$ corresponding to an LDT parameter in  $\params$. 

Given that the deviation measure functions obtained in this chapter have the form $\mesf (t) \equiv e^{- t^b}$, the modulus of continuity is $\omega (h) = e^{- c \, [ \log (1/h) ]^b}$, i.e. we obtain weak-H\"older continuity.

The statements on the Oseledets filtration and decomposition follow directly from the corresponding general criteria. 
\end{proof}

\section{Refinements in the one-variable case}
\label{refinements_1var}
Let us consider now the case of a {\em one}-variable torus translation by a frequency $\om$ that satisfies a {\em stronger} Diophantine condition, namely:
\begin{equation}\label{strongDC}
\norm{k \om} \ge \frac{t}{\abs{k} \, (\log \, \abs{k})^{1+}}
\end{equation}
for some $t > 0$ and for all $k \in \Z \setminus \{0\}$.

Using the already established fiber-LDT estimate in Theorem~\ref{LDTsing-thm} and the corresponding continuity result for Lyapunov exponents, we derive a {\em sharper} fiber-LDT estimate under the additional assumption that the cocycle $A$ has the property $L_1 (A) > L_2 (A)$. This in turn leads to a {\em stronger} modulus of continuity of the LE, the Oseledets filtration and the Oseledets decomposition. 

We note that our argument for proving this sharper fiber-LDT does require the gap condition $L_1 (A) > L_2 (A)$, since it depends essentially on the Avalanche Principle proven in Chapter 2 of our monograph~\cite{LEbook}, see also~\cite{LEbook-chap2}. In fact, the argument requires the full strength of the continuity results in Chapter 3 of~\cite{LEbook} (or of ~\cite{LEbook-chap3}), specifically the finite scale uniform estimates in Lemma 3.4 in~\cite{LEbook} (or Lemma 5.1 in~\cite{LEbook-chap3}).

We remind the reader the particular estimate in the general AP which we use here (see Proposition 2.37 in~\cite{LEbook} or  Proposition 4.2 in~\cite{LEbook-chap2}). We recall the notation $\rgap (g) = \frac{s_1 (g)}{s_2(g)} \in [1, \infty]$ for the ratio of the first two singular values of a matrix $g \in \Mat (m, \R)$. 


\begin{proposition} \label{AP-general}
There exists $c>0$ such that
 given $0<\varepsilon<1$,  $0<\varkappa\leq c\,\varepsilon^ 2$ 
and  \,  $g_0, g_1,\ldots, g_{n-1}\in\gl(m,\R)$, \,
 if
\begin{align*}
 & \rgap (g_i) >  \frac{1}{\varkappa} &  \text{for all }  \  0 \le i \le n-1  
\\
 & \frac{\norm{ g_i \cdot g_{i-1} }}{\norm{g_i}  \, \norm{ g_{i-1}}}  >  \varepsilon  & 
 \text{for all }     \ 1 \le i \le n-1  
\end{align*}
then  
\begin{align*} 
&  \sabs{ \log \norm{ g^{(n)} } + \sum_{i=1}^{n-2} \log \norm{g_i} -  \sum_{i=1}^{n-1} \log \norm{ g_i \cdot g_{i-1}} } \less n \cdot \frac{\varkappa}{\varepsilon^2} \;.
\end{align*}
 \end{proposition}


The argument we are about to present works only for one-variable translations because it requires the following  {\em sharp} version of the quantitative Birkhoff ergodic Theorem~\ref{qBET-thm}, and this sharp version is only available in the one-variable setting. 

\begin{proposition}\label{sharp-qBET}
Let $\om \in \T$ satisfying the strong Diophantine condition \eqref{strongDC} for some $t > 0$ and let $u \colon \strip_r \to [- \infty, \infty)$ be a subharmonic function satisfying the bound
\begin{equation*}
\sup_{z \in \strip_r} u (z) +    \norml{u}{\Lp{2}}   \le C\,.
\end{equation*}

There are constants $c_1, c_2 > 0$ that depend only on $C$ and $r$ and there is $n_0 \in \N$ that depends on $t$ such that for all $\ep > 0$ and $n \ge n_0$ we have:
\begin{equation*}
\abs{ \{ x \in \T \colon \abs{ \frac{1}{n} \sum_{j=0}^{n-1} \, u (x + j \om) - \avg{u} } > \ep \} } < e^{- c_1 \ep n + c_2 (\log n)^4}.
\end{equation*}
\end{proposition}

This result was proven in \cite{GS-Holder} (see Theorem 3.8) for {\em bounded} subharmonic functions. It remains valid in our setting, by the same argument following Lemma~\ref{decay-BMO-1var}.

We can now phrase the sharper fiber-LDT in the one-variable case.

\begin{theorem}\label{sharp-LDT-1var}
Let $A \in \cocyclesone$ with $\det [A(x)] \not\equiv 0$, and let $\om \in \T$ be a frequency satisfying the strong Diophantine condition \eqref{strongDC}. 

Assume that $L_1 (A) > L_2 (A)$ and let $\ep > 0$. There are $\delta = \delta (A) > 0$, $p = p(A) < \infty$, $\bar{n}_1 = \bar{n}_1 (A, \om, \ep) \in \N $ such that for all $B \in \cocyclesone$ with $\norm{B-A}_r < \delta$ and for all $n \ge \bar{n}_1$ we have:
\begin{equation*}
\sabs{ \{ x \in \T \colon \abs{ \frac{1}{n} \, \log \, \norm{\Bn{n} (x) } - \LE{n}_1 (B) } > \ep \} } < e^{- \ep \, n / (\log n)^p}.
\end{equation*}
\end{theorem}

\smallskip

Before starting the proof of  this sharper fiber-LDT, we note that it cannot be obtained along the same lines as Theorem~\ref{LDTsing-thm}, by using the sharper estimate in Proposition~\ref{sharp-qBET}, precisely because the nearly almost invariance property (Proposition~\ref{unif-a-inv}) is too weak, 
as a consequence of the {\em singularities} of the cocycle.

\begin{proof}
We first explain the mechanics of the proof, then we detail the argument, which bears some similarities with an argument used in \cite{GS-fineIDS}.

We already have a version of the fiber-LDT estimate in Theorem~\ref{LDTsing-thm}, where the deviation functions $\devf (t) \equiv \ldtmeas \, t^{- a}$, $\mesf (t) \equiv e^{- t^b}$ are both relatively coarse. Using this LDT at a scale $n_0$ and the avalanche principle in Proposition~\ref{AP-general}, we derive an LDT estimate at a larger scale $n_1 \asymp e^{n_0^{b_1}}$ (where $0 < b_1 < b$), which has a much sharper deviation size function $\devf (t)$, but a very coarse deviation measure function $\mesf (t)$. This will lead, via Lemma~\ref{decay-BMO-1var}, to BMO estimates, which with the help of John-Nirenberg's inequality will prove the desired stronger LDT estimate. 

Let $\ga := L_1 (A) - L_2 (A) > 0$. Let $\bar{n}_0 \in \N$, $\delta > 0$ be such that the fiber-LDT estimate in Theorem~\ref{LDTsing-thm} holds for all cocycles $B \in \cocyclesone$ with $\norm{B-A}_r < \delta$ and for all $n \ge \bar{n}_0$.

Moreover, the Lyapunov exponents are already known to be continuous, and in fact, the more precise finite scale uniform estimates from Lemma 5.1 in~\cite{LEbook-chap3} hold as well. Therefore, we can choose $\bar{n}_0$ and  $\delta$ so that the following 
conditions hold for all  $B \in \cocyclesone$ with $\norm{B-A}_r < \delta$.
Firstly,
\begin{equation}\label{eq1-refine}
\frac{1}{m} \log \, \rgap (\Bn{m} (x) ) \ge \gamma/2
\end{equation}
holds for all $x$ outside a set of measure $ < e^{- m^b}$ and for all $m \ge \bar{n}_0$. 
Also
\begin{equation}\label{eq3-refine}
\abs{ \LE{m_1}_1 (B) - \LE{m_2}_1 (B) } < \gamma/30
\end{equation}
for all $m_1, m_2  \ge \bar{n}_0$.

Fix {\em any} integer $n_0 \ge \bar{n}_0$. We will use the notation $m \asymp n_0$ for scales $m$ such that $n_0 \le m \le 3 n_0$. 

When it comes to other quantities $a$ and $b$, $a \asymp b$ simply means $K^{-1} \, a \le b \le K \, a$, where the constant $K$ is context-universal.

Using \eqref{eq3-refine} and the LDT estimate~\eqref{LDTsing-star} applied at scales $m_1, m_2, m_1+m_2$, it follows from Lemma 4.2 in~\cite{LEbook-chap3} that if $m_1, m_2 \asymp n_0$, then
\begin{equation}\label{eq4-refine}
\frac{\norm{\Bn{m_2+m_1} (x)}}{ \norm{ \Bn{m_2} (T^{m_1} x)  } \cdot \norm{\Bn{m_1} (x) }   }  \ge e^{- n_0 \, \gamma / 5} =: \varepsilon
\end{equation}
holds for all $x$ outside a set of measure $ \less e^{- n_0^b}$.

Moreover, from \eqref{eq1-refine} we have, for $m \asymp n_0$
\begin{equation}\label{eq5-refine}
\rgap (\Bn{m} (x) ) \ge e^{n_0 \, \gamma / 2} =: \frac{1}{\varkappa},
\end{equation}
for all $x$ outside a set of measure $ < e^{- n_0^b}$.

Note also that 
\begin{equation}\label{eq6-refine}
\frac{\varkappa}{\varepsilon^2} = e^{- n_0 \, \gamma / 10} \ll 1\,.
\end{equation}

Let $0 < b_1 < b$ and let $n_1 \asymp e^{n_0^{b_1}}$, hence $n_0 \asymp (\log n_1)^{1/b_1}$. 

Consider the $n_1$-st iterate $\Bn{n_1} (x) = B (T^{n_1-1} \, ) \cdot \ldots \cdot B (T x) \cdot B (x)$.

We will break down this block (i.e. product of matrices) of length $n_1$ into (different configurations of) blocks of length $ \asymp n_0$, and apply the  avalanche principle in Proposition~\ref{AP-general} to the resulting configurations. The estimates \eqref{eq4-refine},  \eqref{eq5-refine}, \eqref{eq6-refine} will ensure that the hypotheses of the avalanche principle are satisfied outside a relatively small set of phases.

Pick any $n \asymp \frac{n_1}{n_0}$ many integers $m_0, m_1, \ldots, m_{n-1}$ such that 
$$m_i \asymp n_0 \ \text{ and } \ n_1 = \sum_{i=0}^{n-1} \, m_i\,.$$

Let $q_0 := 0$ and $q_i := m_{i-1} + \ldots + m_0 = m_{i-1} + q_{i-1}$ for $1 \le i \le n-1$.

\smallskip

Define $g_i = g_i (x) := \Bn{m_i} (T^{q_i} x )$ for $0 \le i \le n-1$.

Then $g^{(n)} = g_{n-1} \,   \ldots   \, g_0 = \Bn{n_1} (x)$ and $g_i   g_{i-1} = \Bn{m_i + m_{i-1}} (T^{q_{i-1}} x)$.

Using \eqref{eq4-refine}, if $x$ is outside a set of measure $ < e^{- n_0^b}$ we have
\begin{equation}\label{eq10-refine}
\frac{\norm{g_i   g_{i-1}}}{\norm{g_i} \,  \norm{g_{i-1}}} = 
\frac{\norm{\Bn{m_i+m_{i-1}} (T^{q_{i-1}} x)}}{ \norm{ \Bn{m_i} (T^{m_{i-1}} \, T^{q_{i-1}} x)  }  \, \norm{\Bn{m_{i-1}} ( T^{q_{i-1}} x) }   }  >  \varepsilon\,,
\end{equation}

while from  \eqref{eq5-refine}, for a similar set,
\begin{equation}\label{eq11-refine}
\rgap (g_i) = \rgap ( \Bn{m_i} (T^{q_i} x ) ) > \frac{1}{\varkappa}\,.
\end{equation}

By excluding a set of measure  $ < n \, e^{- n_0^{b}} < e^{- 1/2 \, n_0^{b}} $ the estimates \eqref{eq10-refine} and \eqref{eq11-refine} will hold for all indices $i$, so the avalanche principle of Proposition~\ref{AP-general} applies and we have:
\begin{align*} 
&  \sabs{ \log \norm{ g^{(n)} } + \sum_{i=1}^{n-2} \log \norm{g_i} -  \sum_{i=1}^{n-1} \log \norm{ g_i  \, g_{i-1}} } \less n \cdot \frac{\ka}{\ep^2} 
< n \, e^{- n_0 \, \gamma/10},
\end{align*}
which becomes
\begin{align}
\Bigl|  \, \log \norm{\Bn{n_1} (x) }    & +  \sum_{i=1}^{n-2} \log \norm{ \Bn{m_i} (T^{q_i} x )  } \label{eq12-refine}  \\ 
& -  \sum_{i=1}^{n-1} \log \norm{ \Bn{m_i + m_{i-1}} (T^{q_{i-1}} x) }   \, \Bigr|    
 < n \, e^{- n_0 \, \gamma/10} \notag\;. 
\end{align}

We will compute an average of~ \eqref{eq12-refine}  to establish a relation  
between the finite scale Lyapunov exponents
$\LE{n_1}_1 (B)$ , $\LE{n_0}_1 (B)$ and $\LE{2 n_0}_1 (B)$ (see formula~ \eqref{eq22-refine} below). In this average,
 the first and the last terms appearing in the second sum of~ \eqref{eq12-refine} will be discarded. Let us now explain why these two terms are  negligible. 

First note that if $m \asymp n_0$, then applying \eqref{LDTsing-star}, we have that for $x$ outside a set of measure $ < e^{- n_0^{b}}$,
$$\frac{1}{m} \, \log \, \norm{\Bn{m} (x) } > \LE{m}_1 (B) - m^{- a} > - C\,,$$
for some  $C = C(A) < \infty$, where the lower bound on $\LE{m}_1 (B)$ follows from \eqref{ulbounds-4}.

Moreover, since by \eqref{ulbounds-3} we have that for all $x \in \T$,
$$\frac{1}{m} \, \log \, \norm{\Bn{m} (x) }  \le C\,,$$
we conclude that for $x$ outside a set of measure $ < e^{- n_0^b}$, we have
$$ \abs{  \log \, \norm{\Bn{m} (x) }  } \le C \, m \less C \, n_0\,.$$

This estimate clearly applies to the functions $x \mapsto  \log \, \norm{\Bn{m_2+m_1} (x) }$ and  $x \mapsto  \log \, \norm{\Bn{m_{n-1} + m_{n-2}} (T^{q_{n-2}} x) } $ which represent the terms corresponding to $i=1$ and $i=n-1$ in the second sum in \eqref{eq12-refine}.

Let $v (x) :=  \log \, \norm{\Bn{m_2+m_1} (x) } +  \log \, \norm{\Bn{m_{n-1} + m_{n-2}} (T^{q_{n-2}} x) } $. Then for $x$ outside a set of measure $ \less e^{- n_0^{b}}$, the function $v (x)$ has the bound
$$\abs{ v (x) } \lesssim C n_0\;. $$

Since $n_1 \asymp e^{n_0^{b_1}}$, with $b_1 < 1$, 
estimate $\eqref{eq12-refine}$ then implies that for all $x$ outside a set of measure $ \less e^{- 1/2 \, n_0^{b}} $ we have:
\begin{align}
& \Bigl|  \, \frac{1}{n_1} \, \log \norm{\Bn{n_1} (x) }    +  \frac{1}{n_1} \, \sum_{i=1}^{n-2} \log \norm{ \Bn{m_i} (T^{q_i} x )  } \label{eq13-refine}  \\ 
& -  \frac{1}{n_1} \, \sum_{i=2}^{n-2} \log \norm{ \Bn{m_i + m_{i-1}} (T^{q_{i-1}} x) }   \, \Bigr|    
 < n \, e^{- n_0 \, \gamma/10} + C \frac{n_0}{n_1} \less \frac{n_0}{n_1} \notag\;.
\end{align}
We are now ready to average~\eqref{eq13-refine} over some specific choices of the integers $m_0, m_1, \ldots, m_{n-1}$. 

The goal is to apply Proposition~\ref{sharp-qBET} to $\un{m}{B} (z) = \frac{1}{m} \,  \log \, \norm{\Bn{m} (z) } $ for $m = n_0$ and $m = 2 n_0$, which will allow us to replace the two sums in \eqref{eq13-refine} by $\LE{n_0}_1 (B)$ and $\LE{2 n_0}_1 (B)$ respectively.

The reason we will not simply choose all integers $m_i$ to be $n_0$, but instead we will consider $n_0$-many configurations and then average, is that otherwise the translates $T^{q_i} \, x = x + q_i \, \om$ would only be by multiples of $n_0$, i.e. $q_i = i \, n_0$. However, Proposition~\ref{sharp-qBET} involves all translations $T^j x = x + j \om$, $ 0 \le j < n_1$ and not just the translations by  $i \, n_0 \, \om$, $0 \le i \le n-1$.


Firstly, pick all integers $m_i$ to be $n_0$ except for the last one, $m_{n-1}$, which is chosen such that $2 n_0 \le m_{n-1} \le 3 n_0$. Note that in this case $q_0 = 0$, $q_i = i \, n_0$ for $1 \le i \le n-1$, hence we are getting the translates by multiples of $n_0$.

Secondly, increase by $1$ the size of the first block, decrease by $1$ the size of the last block, and keep all the other blocks the same. In other words, let $m_0 = n_0 + 1$, $m_1 = m_2 = \ldots = m_{n-2} = n_0$ and $m_{n-1} \asymp n_0$ is chosen so that all integers $m_i$ add up to $n_1$. In this case $q_0 = 0$, $q_i = i n_0 + 1$ for $1 \le i \le n-1$, so we are getting the translates by multiples of $n_0$ plus $1$.

Continue to increase the first block by $1$, decrease the last by $1$ and keep the rest the same, for $n_0$ steps. 

In other words, for each $0 \le j \le n_0 - 1$, choose the following integers:
$$m_0 = n_0 + j, \ m_1 = m_2 = \ldots = m_{n-2} = n_0  \, \text{ and } m_{n-1} \asymp n_0\,,$$
so that they all add up to $n_1$.
Then $q_0 = 0$ and $q_i = i n_0 + j$.

Apply~\eqref{eq13-refine} for each of these $n_0$ configurations of integers $m_i$, $0\le i \le n-1$, add up all the estimates and divide by $n_0$ to get:
\begin{align}
& \Bigl|  \, \frac{1}{n_1} \, \log \norm{\Bn{n_1} (x) }    +  \frac{1}{n_1} \, \sum_{j=0}^{n_0-1} \, \sum_{i=1}^{n-2} \, \frac{1}{n_0} \, \log \norm{ \Bn{n_0} (T^{i \, n_0 + j} x )  } \label{eq14-refine}  \\ 
& -  \frac{2}{n_1} \, \sum_{j=0}^{n_0-1} \, \sum_{i=2}^{n-2} \, \frac{1}{2 n_0} \, \log \norm{ \Bn{2 n_0} (T^{(i-1) \, n_0 + j} x) }   \, \Bigr|    
 <  C \frac{n_0}{n_1}  \notag
\end{align}
for all $x$ outside a set of measure $ < n_0\, e^{- 1/2 \, n_0^{b}} $.

Estimate~\eqref{eq14-refine} can be written as:
\begin{align}
& \Bigl|  \, \frac{1}{n_1} \, \log \norm{\Bn{n_1} (x) }    +  \frac{1}{n_1}  \sum_{k=n_0}^{(n-1) \, n_0 - 1}  \frac{1}{n_0} \, \log \norm{ \Bn{n_0} (T^{k} x )  } \label{eq15-refine}  \\ 
& -  \frac{2}{n_1}   \sum_{k=n_0}^{(n-2) \, n_0 - 1}  \frac{1}{2 n_0} \, \log \norm{ \Bn{2 n_0} (T^k  x) }   \, \Bigr|    
 <  C \frac{n_0}{n_1}  \asymp \frac{(\log n_1)^{1/b_1}}{n_1}   \notag
\end{align}
for all $x$ outside a set of measure $ < e^{- 1/3 \, n_0^{b}} $.

Due to the uniform estimates \eqref{ulbounds-3}, \eqref{ulbounds-4} in Proposition~\ref{uplowbounds}, Proposition~\ref{sharp-qBET} can be applied to the functions
$$\un{n_0}{B} (z) = \frac{1}{n_0} \,  \log \, \norm{\Bn{n_0} (z) } \ \text{ and } \ \un{2 n_0}{B} (z) = \frac{1}{2 n_0} \,  \log \, \norm{\Bn{n_0} (z)}\,.$$

Let $\ep \asymp \frac{(\log n_1)^{1/b_1}}{n_1}$ and we may of course assume that  $1/b_1 > 1$.
Moreover, since the number $(n-2) \, n_0$ and respectively $(n-3) \, n_0$ of translates  in \eqref{eq15-refine} are $ \asymp n_1$, then up to an additional error of order $\frac{n_0}{n_1}$, using Proposition~\ref{sharp-qBET} we obtain:

\begin{equation}\label{eq20-refine}
\Bigl| \frac{1}{n_1}  \sum_{k=n_0}^{(n-1) \, n_0 - 1}  \frac{1}{n_0} \, \log \norm{ \Bn{n_0} (T^{k} x )  } - \LE{n_0}_1 (B) \Bigr| < \ep  \asymp \frac{(\log n_1)^{1/b_1}}{n_1}
\end{equation} 
for $x$ outside a set of measure 
$ < e^{- c_1 \ep n_1 + c_2 (\log n_1)^4} < e^{- c_3 \, (\log n_1)^{1/b_1}}$.

Similarly we have:
\begin{equation}\label{eq21-refine}
\Bigl| \frac{1}{n_1}  \sum_{k=n_0}^{(n-2) \, n_0 - 1}  \frac{1}{2 n_0} \, \log \norm{ \Bn{2 n_0} (T^{k} x )  } - \LE{2 n_0}_1 (B) \Bigr| <  \frac{(\log n_1)^{1/b_1}}{n_1}
\end{equation} 
for $x$ outside a set of measure 
$ < e^{- c_3 \, (\log n_1)^{1/b_1}}$.

Now integrate \eqref{eq15-refine} to get:
\begin{align}
\abs{  \LE{n_1}_1 (B)  &+  \LE{n_0}_1 (B) - 2 \LE{2 n_0}_1 (B)  }  \notag \\
  & < C \frac{n_0}{n_1}  + C \, e^{- 1/3 \, n_0^b} \less  \frac{(\log n_1)^{1/b_1}}{n_1} \;. \label{eq22-refine}
\end{align}

Combine \eqref{eq15-refine}, \eqref{eq20-refine}, \eqref{eq21-refine}, \eqref{eq22-refine} to conclude that for $x$ outside a set of measure $ \less e^{- c_3 \, (\log n_1)^{1/b_1}}$, we have
\begin{equation}\label{eq23-refine}
\sabs{   \frac{1}{n_1} \, \log \norm{\Bn{n_1} (x) }   - \LE{n_1}_1 (B)  } \less \frac{(\log n_1)^{1/b_1}}{n_1}\;.
\end{equation}

Let $$\ep_0 \asymp \frac{(\log n_1)^{1/b_1}}{n_1}, \quad \ep_1 \asymp e^{- c_3 \, (\log n_1)^{1/b_1}} \ll  \ep_0^8  $$ 
and let $$ u (z) = \un{n_1}{B} (z) = \frac{1}{n_1} \,  \log \, \norm{\Bn{n_1} (z) }\;. $$

Then $u (z)$ is subharmonic on $\strip_r$, and by Proposition~\ref{uplowbounds}, for some $ C = C(A) < \infty$  we have the bound
$$\sup_{z \in \strip_r} u (z) + \norml{u}{\Lp{2}} \le C\,,$$
which via Remark~\ref{decays-sh:remark} implies \eqref{measurement-sh} with $\umeas = \mathscr{O} \bigl( \frac{C_r }{r} \, C \bigr)$.

Moreover, from \eqref{eq23-refine}
$$\abs{ \{ x \in \T \colon \abs{ u (x) - \avg{u} } > \ep_0 \} } < \ep_1\,.$$

All the assumptions of Lemma~\ref{decay-BMO-1var} are then satisfied and we conclude that
$$\norml{u}{BMO (\T)} \less \ep_0 + ( \umeas \, \ep_1 )^{1/2} \less \ep_0\,,$$
provided $n_1$ is large enough, depending on $A$, so that $\ep_1$ absorbs the constant $\umeas$.

Therefore, 
$$\norml{u}{BMO (\T)} \less \ep_0 \asymp \frac{(\log n_1)^{1/b_1}}{n_1}\,,$$
and so John-Nirenberg's inequality implies that for all $ \ep > 0$
$$\abs{ \{ x \in \T \colon \abs{ u (x) - \avg{u} } > \ep \} } < e^{- c_0 \, \ep / \norml{u}{BMO (\T)}}.$$

Writing the last estimate in terms of iterates of our cocycle we have:
\begin{equation}\label{eq30-refine}
\sabs{  \{ x \in \T \colon  \abs{ \frac{1}{n_1} \, \log \norm{\Bn{n_1} (x) }   - \LE{n_1}_1 (B) } > \ep \} } < e^{- c_0 \, \ep \, n_1 / (\log n_1)^{1/b_1} }.
\end{equation}

We ran this argument starting at {\em any} scale $n_0 \ge \bar{n}_0$ and we obtained \eqref{eq30-refine} for $n_1 \asymp e^{n_0^{b_1}}$. Therefore, if 
$\bar{n}_1 \asymp e^{\bar{n}_0^{b_1}}$, then the estimate \eqref{eq30-refine} holds for all $n_1 \ge \bar{n}_1$, which proves our theorem.
\end{proof}

\begin{remark}\label{mod-cont-strong-one-var} \normalfont
Using the terminology in Section~\ref{cont_le_qp}, we have shown that for a one-variable torus translation by a (strongly) Diophantine frequency, a uniform fiber-LDT holds with deviation measure function
$$\mesf (t) \equiv e^{- c \, t / [ \log t ]^b}, \quad \text{ for some } c, b > 0\,.$$

We may then conclude, as in the proof of Theorem~\ref{qp_ldt_thm:cont}, that if $A\in \cocycles_{m}$ has a $\tau$-gap pattern, then in a neighborhood of $A$ the functions
$\Lambda^\tau$, $\filt^\tau$ and $\dec^\tau$ have the modulus of continuity 
$$\omega (h) :=  e^{- c \, [ \log 1/h ] / [ \log  \log 1/h ]^b}.$$

That is, in the presence of gaps in the Lyapunov spectrum, the LE, the Oseledets filtration and the Oseledets decomposition are locally nearly-H\"older continuous.

 \end{remark}

\bigskip

\subsection*{Acknowledgments}

The first author was supported by 
Funda\c{c}\~{a}o  para a  Ci\^{e}ncia e a Tecnologia, 
UID/MAT/04561/2013.

The second author was supported by the Norwegian Research Council project no. 213638, "Discrete Models in Mathematical Analysis".  

The second author would like to thank C. Marx for suggesting the problem and M. Vod\u{a} for a useful conversation on this topic.

\bigskip

\bibliographystyle{amsplain} 

\providecommand{\bysame}{\leavevmode\hbox to3em{\hrulefill}\thinspace}
\providecommand{\MR}{\relax\ifhmode\unskip\space\fi MR }
\providecommand{\MRhref}[2]{%
  \href{http://www.ams.org/mathscinet-getitem?mr=#1}{#2}
}
\providecommand{\href}[2]{#2}
\begin{thebibliography}{10}

\bibitem{AJS}
Artur \'{A}vila, Svetlana Jitomirskaya, and Christian Sadel, \emph{Complex
  one-frequency cocycles}, J. Eur. Math. Soc. (JEMS) \textbf{16} (2014), no.~9,
  1915--1935. \MR{3273312}

\bibitem{B}
J.~Bourgain, \emph{Green's function estimates for lattice {S}chr\"odinger
  operators and applications}, Annals of Mathematics Studies, vol. 158,
  Princeton University Press, Princeton, NJ, 2005. \MR{2100420 (2005j:35184)}

\bibitem{B-d}
\bysame, \emph{Positivity and continuity of the {L}yapounov exponent for shifts
  on {$\Bbb T^d$} with arbitrary frequency vector and real analytic potential},
  J. Anal. Math. \textbf{96} (2005), 313--355. \MR{2177191 (2006i:47064)}

\bibitem{B-J}
J.~Bourgain and S.~Jitomirskaya, \emph{Continuity of the {L}yapunov exponent
  for quasiperiodic operators with analytic potential}, J. Statist. Phys.
  \textbf{108} (2002), no.~5-6, 1203--1218, Dedicated to David Ruelle and Yasha
  Sinai on the occasion of their 65th birthdays. \MR{1933451 (2004c:47073)}

\bibitem{DK2}
Pedro Duarte and Silvius Klein, \emph{Continuity of the {L}yapunov {E}xponents
  for {Q}uasiperiodic {C}ocycles}, Comm. Math. Phys. \textbf{332} (2014),
  no.~3, 1113--1166. \MR{3262622}

\bibitem{LEbook-chap3}
\bysame, \emph{An abstract continuity theorem of the {L}yapunov exponents},
  preprint (2015), 1--32.

\bibitem{LEbook-chap2}
\bysame, \emph{The {A}valanche {P}rinciple and other estimates on {G}rassmann
  manifolds}, preprint (2015), 1--49.

\bibitem{LEbook-chap4}
\bysame, \emph{Continuity of the {O}seledets decomposition}, preprint (2015),
  1--32.

\bibitem{LEbook}
\bysame, \emph{Lyapunov exponents of linear cocycles, continuity via large
  deviations}, Atlantis Studies in Dynamical Systems, Atlantis Press, to appear
  in 2016.

\bibitem{Duoa}
Javier Duoandikoetxea, \emph{Fourier analysis}, Graduate Studies in
  Mathematics, vol.~29, American Mathematical Society, Providence, RI, 2001,
  Translated and revised from the 1995 Spanish original by David Cruz-Uribe.
  \MR{1800316 (2001k:42001)}

\bibitem{GS-Holder}
Michael Goldstein and Wilhelm Schlag, \emph{H\"older continuity of the
  integrated density of states for quasi-periodic {S}chr\"odinger equations and
  averages of shifts of subharmonic functions}, Ann. of Math. (2) \textbf{154}
  (2001), no.~1, 155--203. \MR{1847592 (2002h:82055)}

\bibitem{GS-fineIDS}
\bysame, \emph{Fine properties of the integrated density of states and a
  quantitative separation property of the {D}irichlet eigenvalues}, Geom.
  Funct. Anal. \textbf{18} (2008), no.~3, 755--869. \MR{2438997 (2010h:47063)}

\bibitem{HK-book}
W.~K. Hayman and P.~B. Kennedy, \emph{Subharmonic functions. {V}ol. {I}},
  Academic Press [Harcourt Brace Jovanovich, Publishers], London-New York,
  1976, London Mathematical Society Monographs, No. 9. \MR{0460672 (57 \#665)}

\bibitem{JitMarx}
S.~Jitomirskaya and C.~A. Marx, \emph{Continuity of the {L}yapunov exponent for
  analytic quasi-periodic cocycles with singularities}, J. Fixed Point Theory
  Appl. \textbf{10} (2011), no.~1, 129--146. \MR{2825743 (2012h:37095)}

\bibitem{JitMarx-CMP}
\bysame, \emph{Analytic quasi-perodic cocycles with singularities and the
  {L}yapunov exponent of extended {H}arper's model}, Comm. Math. Phys.
  \textbf{316} (2012), no.~1, 237--267. \MR{2989459}

\bibitem{sK1}
Silvius Klein, \emph{Anderson localization for the discrete one-dimensional
  quasi-periodic {S}chr\"odinger operator with potential defined by a
  {G}evrey-class function}, J. Funct. Anal. \textbf{218} (2005), no.~2,
  255--292. \MR{2108112 (2005m:82070)}

\bibitem{sK2}
\bysame, \emph{Localization for quasiperiodic {S}chr\"{o}dinger operators with
  multivariable {G}evrey potential functions}, J. Spectr. Theory \textbf{4}
  (2014), 1--53.

\bibitem{Levin}
B.~Ya. Levin, \emph{Lectures on entire functions}, Translations of Mathematical
  Monographs, vol. 150, American Mathematical Society, Providence, RI, 1996, In
  collaboration with and with a preface by Yu. Lyubarskii, M. Sodin and V.
  Tkachenko, Translated from the Russian manuscript by Tkachenko. \MR{1400006
  (97j:30001)}

\bibitem{Muscalu-S}
Camil Muscalu and Wilhelm Schlag, \emph{Classical and multilinear harmonic
  analysis. {V}ol. {I}}, Cambridge Studies in Advanced Mathematics, vol. 137,
  Cambridge University Press, Cambridge, 2013. \MR{3052498}

\bibitem{S56}
Ernst~S. Selmer, \emph{On the irreducibility of certain trinomials}, Math.
  Scand. \textbf{4} (1956), 287--302. \MR{0085223 (19,7f)}

\bibitem{KTao-d}
Kai Tao, \emph{Continuity of {L}yapunov exponent for analytic quasi-periodic
  cocycles on higher-dimensional torus}, Front. Math. China \textbf{7} (2012),
  no.~3, 521--542. \MR{2915794}

\bibitem{KTao-Jacobi}
\bysame, \emph{H\"older continuity of {L}yapunov exponent for quasi-periodic
  {J}acobi operators}, Bull. Soc. Math. France \textbf{142} (2014), no.~4,
  635--671. \MR{3306872}

\bibitem{WangZhang-cont-C2}
Y.~Wang and Z.~Zhang, \emph{Uniform positivity and continuity of {L}yapunov
  exponents for a class of $c^2$ quasiperiodic schr\"odinger cocycles},
  preprint (2013), 1--49.

\bibitem{YouZhang-Holder-cont}
Jiangong You and Shiwen Zhang, \emph{H\"older continuity of the {L}yapunov
  exponent for analytic quasiperiodic {S}chr\"odinger cocycle with weak
  {L}iouville frequency}, Ergodic Theory Dynam. Systems \textbf{34} (2014),
  no.~4, 1395--1408. \MR{3227161}

\end{thebibliography}

\providecommand{\bysame}{\leavevmode\hbox to3em{\hrulefill}\thinspace}
\providecommand{\MR}{\relax\ifhmode\unskip\space\fi MR }
\providecommand{\MRhref}[2]{%
  \href{http://www.ams.org/mathscinet-getitem?mr=#1}{#2}
}
\providecommand{\href}[2]{#2}

\end{document}